\documentclass[10pt]{article}

\usepackage{listings}
\usepackage{pdflscape}
\usepackage{epsfig}
\usepackage{lscape}
\usepackage{longtable}
\usepackage{amsmath,amssymb,amsthm}
\usepackage{graphicx}
\usepackage{color}
\usepackage{placeins}
\usepackage{url}
\usepackage{cases}
\usepackage[ruled]{algorithm2e}
\usepackage[scaled=0.9]{helvet}
\numberwithin{equation}{section}
\usepackage{abstract}
\usepackage[title,titletoc,header]{appendix}
\usepackage{multirow}
\usepackage{subfigure}
\usepackage{float}
\makeatletter\def\@captype{table}\makeatother

\def\sqr#1#2{{\vcenter{\vbox{\hrule height.#2pt
        \hbox{\vrule width.#2pt height#1pt \kern#2pt
        \vrule width.#2pt}
        \hrule height.#2pt}}}}
\def\approxleq{ \kern3pt \mbox{\raisebox{.6ex}{$<$}} \kern-8pt
  \mbox{\raisebox{-.6ex}{$\sim$}} \kern5pt}

\newlength{\len}
\usepackage[dvipdfm]{geometry}
\usepackage{setspace}
\if@onethmnum
  \newtheorem{theorem}{Theorem}
  \newtheorem{lemma}[theorem]{Lemma}
  \newtheorem{corollary}[theorem]{Corollary}
  \newtheorem{proposition}[theorem]{Proposition}
  \newtheorem{definition}[theorem]{Definition}
  \newtheorem{assumption}[theorem]{Assumption}
  \newtheorem{remark}[theorem]{Remark}
  \newtheorem{property}[theorem]{Property}
  \newtheorem{example}[theorem]{Example}

\else
  \newtheorem{theorem}{Theorem}[section]
  \newtheorem{lemma}[theorem]{Lemma}
  
  \newtheorem{proposition}[theorem]{Proposition}
  
  \newtheorem{assumption}[theorem]{Assumption}
  \newtheorem{remark}[theorem]{Remark}
  
  \newtheorem{example}[theorem]{Example}
\fi

\begin{document}

\title{A two-phase strategy for control constrained elliptic optimal control problems}

\author{Xiaoliang Song\thanks{School of Mathematical Sciences,
Dalian University of Technology, Dalian, Liaoning 116025, China
(\tt{songxiaoliang@mail.dlut.edu.cn}). 
}
\and Bo Yu\thanks{Corresponding author. School of Mathematical Sciences, Dalian University of Technology, Dalian, Liaoning 116024, China. (\tt{yubo@dlut.edu.cn}).}
}

\date{\today}
\maketitle

\begin{abstract}
Elliptic optimal control problems with pointwise box constraints on the control (EOCP) are considered. To solve EOCP, the primal-dual active set (PDAS) method, which is a special semismooth Newton (SSN) method, used to be a priority in consideration of their locally superlinear convergence. However, in general solving the Newton equations is expensive, especially when the discretization is in a fine level. Motivated by the success of applying alternating direction method of multipliers (ADMM) for solving large scale convex minimization problem in finite dimension, it is reasonable to extend the ADMM to solve EOCP. To numerically solve EOCP, the finite element (FE) method is used for discretization. Then, a two-phase strategy is presented to solve discretized problems. In Phase-I, an inexact heterogeneous ADMM (ihADMM) is proposed with the aim of solving discretized problems to moderate accuracy or using it to generate a reasonably good initial point to warm-start Phase-II. Different from the classical ADMM, our ihADMM adopts two different weighted inner product to define the augmented Lagrangian function in two subproblems, respectively. Benefiting from such different weighted techniques, two subproblems of ihADMM can be efficiently implemented. Furthermore, theoretical results on the global convergence as well as the iteration complexity results $o(1/k)$ for ihADMM are given. In Phase-II, in order to obtain more accurate solution, the primal-dual active set (PDAS) method is used as a postprocessor of the ihADMM. Numerical results show that the ihADMM and the two-phase strategy are highly efficient.
\end{abstract}


\pagestyle{myheadings} \thispagestyle{plain} \markboth{XIAOLIANG SONG, BO YU, YIYANG WANG}{A TWO-PHASE METHOD FOR OPTIMAL CONTROL PROBLEM}

\section{Introduction}
\vspace{-2pt}
In this paper, we study the following linear-quadratic elliptic PDE-constrained optimal control problems with box constraints on the control:
\begin{equation}\label{eqn:orginal problems}
           \qquad \left\{ \begin{aligned}
        &\min \limits_{(y,u)\in Y\times U}^{}\ \ J(y,u)=\frac{1}{2}\|y-y_d\|_{L^2(\Omega)}^{2}+\frac{\alpha}{2}\|u\|_{L^2(\Omega)}^{2} \\
        &\qquad{\rm s.t.}\qquad Ly=u+y_c\ \ \mathrm{in}\  \Omega, \\
         &\qquad \qquad \qquad y=0\quad  \mathrm{on}\ \partial\Omega,\\
         &\qquad \qquad\qquad u\in U_{ad}=\{v(x)|a\leq v(x)\leq b, {\rm a.e. }\  \mathrm{on}\ \Omega\}\subseteq U,
                          \end{aligned} \right.\tag{$\mathrm{P}$}
 \end{equation}
where $Y:=H_0^1(\Omega)$, $U:=L^2(\Omega)$, $\Omega\subseteq \mathbb{R}^2$ is a convex, open and bounded domain with $C^{1,1}$- or polygonal boundary $\Gamma$; the source term $y_c\in L^2(\Omega)$ and the desired state $y_d \in L^2(\Omega)$; parameters $-\infty<a<b<+\infty$, $\alpha>0$ and the operator $L$ is a second-order linear elliptic differential operator. Such problem (\ref{eqn:orginal problems}) is very important in practical applications, e.g., the electro-magnetic induction with distributed heat source.

Optimization problems with constraints which require the solution of a partial differential equation arise widely in many areas of the sciences and engineering, in particular in problems of design. Optimal control problems with control constraints and their numerical realization have been studied intensively in recent papers, see e.g., \cite{hinze2005variational, falk1973approximation, geveci1979approximation, casas2003error, meyer2004superconvergence} and the references cited there. Let us first comment on known results on error estimates analysis of control constrained optimal control problems. Basic a-priori error estimates were derived by Falk \cite{falk1973approximation} and Geveci \cite{geveci1979approximation} where Falk considered distributed controls, while Geveci concentrates on Neuman boundary controls. Both the authors for piecewise constant control approximations prove optimal $L^2$-error estimates $\mathcal{O}(h)$. Meanwhile, we refer to Arada and Raymond \cite{arada2002error}, where the authors contributed further errors estimate for piecewise constant control approximations and showed convergence of the same order $\mathcal{O}(h)$ in $L^\infty$ norm.
Furthermore, for the approximation of controls by piecewise linear, globally continuous elements, as far as we know, the first error estimates for this type of approximations were proved by Casas and Tr\"{o}ltzsch \cite{casas2003error}in the case of linear-quadratic control problems, proving order $\mathcal{O}(h)$. Later Casas \cite{linearerror} proved order $o(h)$ for control problems governed by semilinear elliptic equations and quite general cost functionals. In \cite{rosch2006error}, R\"{o}sch proved that the error order is $\mathcal{O}(h^{\frac{3}{2}})$ under special assumptions on the continuous solutions, compare also Casas and Tr\"{o}ltzsch \cite{casas2003error}. However, his proof was done for linearquadratic control problems in one dimension. Moreover, a variational discretization concept is introduced by Hinze \cite{hinze2005variational} and a control error of order $\mathcal{O}(h^2)$ is obtained. In certain situations, the same convergence order also can be achieved by a special postprocessing procedure, see Meyer and R\"{o}sch \cite{meyer2004superconvergence}.

Next, let us mention some existing numerical methods for solving problem (\ref{eqn:orginal problems}). As far as we know, most of the aforementioned papers are devoted to directly solve the first-order optimality system, which result in a nonsmooth equation that has to be solved. For this purpose, applying semismooth Newton methods is used to be a priority. A special semismooth Newton method with the active set strategy, called the primal-dual active set (PDAS) method is introduced in \cite{BeItKu} for control constrained elliptic optimal control problems. It is proved to have the locally superlinear convergence (see \cite{Ulbrich2} for more details). Furthermore, mesh-independence results for semismooth Newton methods were established in \cite{meshindependent}. However, in general, it is expensive in solving Newton equations, especially when the discretization is in a fine level.

Recently, for the finite dimensional large scale optimization problems, some efficient first-order algorithms, such as iterative shrinkage/soft thresholding algorithms (ISTA) \cite{Blumen}, accelerated proximal gradient (APG)-based methods \cite{inexactAPG,Beck,inexactABCD}, the ADMM \cite{Boyd,SunToh1,SunToh2,Fazel}, etc., have become the state of the art algorithms. In this paper, we will mainly focus on the ADMM. The classical ADMM was originally proposed by Glowinski and Marroco \cite{GlowMarro} and Gabay and Mercier \cite{Gabay}, and it has found lots of efficient applications. 
In particular, we refer to \cite{Boyd} for a review of the applications of the ADMM in the areas of distributed optimization and statistical learning.

Motivated by above-mentioned facts, we aim to design an algorithm which could solve (\ref{eqn:orginal problems}) to obtain high accurate solution efficiently and fast. To achieve our goal, combining the ADMM and semismooth Newton methods together, a two-phase algorithm is proposed in function space. Specifically, the classical ADMM is developed in Phase-I with the aim of solving an optimal control problem moderate accuracy or using it to generate a reasonably good initial point to warm-start Phase-II. In Phase-II, the PDAS method is used to obtain accurate solutions fast. More importantly, as will be mentioned in Section \ref{sec:2}, each subproblem of the ADMM has a well-formed structure.
Focusing on these inherent structures of the ADMM in function space is worthwhile for us to propose an appropriate discretization scheme and give a suitable algorithm to solve the corresponding discretized problem. Moreover, it will be a crucial point in the numerical analysis to establish similar structures parallel to the ADMM in function space for the discretized problem.

To discretize problem (\ref{eqn:orginal problems}), we consider to use the piecewise linear finite element. Then, in Phase-I, in order to use ADMM-type algorithm to solve the corresponding discretization problem, an artificial variable $z$ should be introduced, and the discretization problem could be equivalently rewritten as a separable form. However, when the classical ADMM is directly used to solve discretized problems, there is no well-formed structure as in continuous case. An important fact is that discretization and algorithms should not be viewed as independent. Rather, they must be intertwined to yield an efficient algorithm for discretized problems. Hence, making use of the inherent structure of problem, an heterogeneous ADMM is proposed. Different from the classical ADMM, we utilize two different weighted inner products to define the augmented Lagrangian function for two subproblems, respectively.
Specifically, based on the $M_h$-weighted inner product, the augmented Lagrangian function with respect to the $u$-subproblem in $k$-th iteration is defined as
\begin{equation*}
  \mathcal{L}_\sigma(u,z^k;\lambda^k)=f(u)+g(z^k)+\langle\lambda,M_h(u-z^k)\rangle+\frac{\sigma}{2}\|u-z^k\|_{M_h}^{2},
\end{equation*}
where $M_h$ is the mass matrix. On the other hand, for the $z$-subproblem, based on the $W_h$-weighted inner product, the augmented Lagrangian function in $k$-th iteration is defined as
\begin{equation*}
  \mathcal{L}_\sigma(u^{k+1},z;\lambda^k)=f(u^{k+1})+g(z)+\langle\lambda,M_h(u^{k+1}-z)\rangle+\frac{\sigma}{2}\|u^{k+1}-z\|_{W_h}^{2},
\end{equation*}
where the lumped mass matrix $W_h$ is diagonal. Furthermore, sometimes it is unnecessary to exactly compute the solution of each subproblem even if it is doable, especially at the early stage of the whole process. For example, if a subproblem is equivalent to solving a large-scale or ill-condition linear system, it is a natural idea to use the iterative methods such as some Krylov-based methods. Hence, taking the inexactness of the solutions of associated subproblems into account, a more practical inexact heterogeneous ADMM (ihADMM) is proposed.

As will be mentioned in the Section \ref{sec:4}, benefiting from different weighted techniques, each subproblem of ihADMM for (\ref{matrix vector form}), i.e., the discrete version of (\ref{eqn:orginal problems}), can be efficiently solved. Specifically, the $u$-subproblem of ihADMM, which result in a large scale linear system, is the main computation cost in whole algorithm. $M_h$-weighted technique could help us to reduce the block three-by-three system to a block two-by-two system without any computational cost so as to reduce calculation amount. On the other hand, $W_h$-weighted technique makes $z$-subproblem have a decoupled form and admit a closed form solution given by the soft thresholding operator and the projection operator onto the box constraint $[a,b]$. Moreover, global convergence and the iteration complexity result $o(1/k)$ in non-ergodic sense for our ihADMM will be proved.

Taking the precision of discretized error into account, we should mention that using our ihADMM algorithm to solve problem (\ref{matrix vector form}) is highly enough and efficient in obtaining an approximate solution with moderate accuracy. Nevertheless, in order to obtain more accurate solutions, if necessarily required, combining ihADMM and semismooth Newton methods together, we give a two-phase strategy. Specifically, our ihADMM algorithm as the Phase-I is used to generate a reasonably good initial point to warm-start Phase-II. In Phase-II, the PDAS method as a postprocessor of our ihADMM is employed to solve the discrete problem to high accuracy.

The remainder of this paper is organized as follows. In Section \ref{sec:2}, a two-phase strategy with an inexact ADMM as Phase-I and PDAS as Phase-II in function space for solving (\ref{eqn:orginal problems}) is described. In Section \ref{sec:3}, the finite element approximation is introduced. In Section \ref{sec:4}, an inexact heterogeneous ADMM is proposed. And as the Phase-II algorithm, the PDAS method is also presented. In Section \ref{sec:5}, numerical results are given to show the efficiency of our ihADMM and two-phase strategy. Finally, we conclude our paper in Section \ref{sec:6}.

Through this paper, let us suppose the elliptic PDEs involved in (\ref{eqn:orginal problems})
\begin{equation}\label{eqn:state equations}
\begin{aligned}
  Ly&=u+y_c \quad \mathrm{in}\  \Omega, \\
   y&=0\qquad \quad \mathrm{on}\ \partial\Omega,
\end{aligned}
\end{equation}
satisfy the following assumption:
\begin{assumption}\label{equ:assumption:1}
The linear second-order differential operator $L$ is defined by
 \begin{equation}\label{operator A}
   (Ly)(x):=-\sum \limits^{n}_{i,j=1}\partial_{x_{j}}(a_{ij}(x)y_{x_{i}})+c_0(x)y(x),
 \end{equation}
where functions $a_{ij}(x), c_0(x)\in L^{\infty}(\Omega)$, $c_0\geq0$,
and it is uniformly elliptic, i.e. $a_{ij}(x)=a_{ji}(x)$ and there is a constant $\theta>0$ such that
\begin{equation}\label{equ:operator A coercivity}
  \sum \limits^{n}_{i,j=1}a_{ij}(x)\xi_i\xi_j\geq\theta\|\xi\|^2 \qquad \mathrm{for\ a.a.}\ x\in \Omega\  \mathrm{and}\  \forall \xi \in \mathbb{R}^n.
\end{equation}
\end{assumption}

The weak formulation of (\ref{eqn:state equations}) is given by
\begin{equation}\label{eqn:weak form}
  \mathrm{Find}\ y\in H_0^1(\Omega):\ a(y,v)=(u+y_c,v)_{L^2(\Omega)}\quad \mathrm{for}\ \forall v \in H_0^1(\Omega)
\end{equation}
with the bilinear form
\begin{equation}\label{eqn:bilinear form}
  a(y,v)=\int_{\Omega}(\sum \limits^{n}_{i,j=1}a_{ji}y_{x_{i}}v_{x_{i}}+c_0yv)\mathrm{d}x.
\end{equation}
Then, utilizing the Lax-Milgram lemma, we have the following proposition.
\begin{proposition}[{\rm \cite [Theorem B.4]{KiSt}}]\label{equ:weak formulation}
Under Assumption {\rm \ref{equ:assumption:1}}, the bilinear form $a(\cdot,\cdot)$ in {\rm (\ref{eqn:bilinear form})} is bounded and $V$-coercive for $V=H^1_0(\Omega)$ and the associate operator $A$ has a bounded inverse. In particular, for a given $y_c\in L^2(\Omega)$ and every $u \in L^2(\Omega)$, {\rm (\ref{eqn:state equations})} has a unique weak solution $y=y(u)\in H^1_0(\Omega)$ given by {\rm (\ref{eqn:weak form})}. Furthermore,
\begin{equation}\label{equ:control estimats state}
  \|y\|_{H^1}\leq C \|u+y_c\|_{L^2(\Omega)},
\end{equation}
for a constant $C$ depending only on $a_{ij}$, $c_0$ and $\Omega$.
\end{proposition}

By Proposition {\rm\ref{equ:weak formulation}}, the solution operator $\mathcal{S}$: $H^{-1}(\Omega)\rightarrow H^1_0(\Omega)$ with $y(u):=\mathcal{S}(u+y_c)$ is well-defined and called the control-to-state mapping, which is a continuous linear injective operator. Since $H_0^1(\Omega)$ is a Hilbert space, the adjoint operator $\mathcal{S^*}$: $H^{-1}(\Omega)\rightarrow H_0^1(\Omega)$ is also a continuous linear operator.
\vspace{-6pt}

\section{A two-phase method in function space}
\label{sec:2}
\vspace{-2pt}
As we have mentioned, the ADMM is a simple but powerful algorithm that is well suited to distributed convex optimization, and in particular to large scale problems arising in applied machine learning and related areas. Motivated by the success of the finite dimensional ADMM algorithm, it is strongly desirable and practically valuable to extend the ADMM to optimal control problems. Moreover, if more accurate solution is necessarily required, combining ADMM and semismooth Newton methods together is a wise choice. Thus, in this section, we will introduce a two-phase strategy. Specifically, an inexact ADMM (iADMM) is developed in Phase-I with the aim of generating a reasonably good initial point to warm-start Phase-II. In Phase-II, the primal-dual active set (PDAS) method is used as a postprocessor of the iADMM.

\subsection{An inexact ADMM in function space}
In this section, we first introduce an inexact ADMM for problem (\ref{eqn:orginal problems}) as the Phase-I algorithm. To obtain a separable form and separate the smooth and nonsmooth terms, by adding an artificial variable $z$, we can equivalently reformulate problem (\ref{eqn:orginal problems}) in as:
\begin{equation}\label{eqn:modified problems}
 \left\{ \begin{aligned}
        &\min \limits_{(y,u,z)\in Y\times U\times U}^{}\ \ \frac{1}{2}\|y-y_d\|_{L^2(\Omega)}^{2}+\frac{\alpha}{2}\|u\|_{L^2(\Omega)}^{2}+\delta_{U_{ad}}(z) \\
        &\qquad\quad{\rm s.t.}\qquad \quad y=\mathcal{S}(u+y_c),\\
         &\qquad \qquad \quad \qquad ~u=z.
                          \end{aligned} \right.\tag{$\mathrm{{DP}}$}
 \end{equation}

It is clear that problem (\ref{eqn:orginal problems}) is strongly convex. Therefore, By the equivalence between (\ref{eqn:orginal problems}) and (\ref{eqn:modified problems}), the existence and uniqueness of solution of (\ref{eqn:modified problems}) is obvious. The optimal solution $(y^*, u^*, z^*)$ can be characterized by the following Karush-Kuhn-Tucker (KKT) conditions.

\begin{theorem}[{\rm First-Order Optimality Condition}]\label{First-Order Optimality Condition}
Under Assumption \ref{equ:assumption:1}, {\rm($y^*$, $u^*$, $z^*$)} is the optimal solution of {\rm(\ref{eqn:modified problems})}, if and only if there exists adjoint state $p^*\in H_0^1(\Omega)$ and Lagrange multiplier $\lambda^*\in L^2(\Omega)$, such that the following conditions hold in the weak sense
\begin{equation}\label{eqn:KKT}
  \left\{\begin{aligned}
        &y^*=\mathcal{S}(u^*+y_c),\\
        &p^*=\mathcal{S}^*(y_d-y^*),\\
&{\alpha} u^*-p^*+\lambda^*=0,\\
&u^*=z^*,\\
&z^* = \Pi_{U_{ad}}(z^*+\lambda^*).
\end{aligned}\right.
\end{equation}
where the projection operator $\mathrm{\Pi}_{U_{ad}}(\cdot)$ is defined as follows:
\begin{equation}\label{projection operator}
 \mathrm{\Pi}_{U_{ad}}(v(x)):=\max\{a,\min\{v(x),b\}\}.
\end{equation}
\end{theorem}
Moreover, by Proposition \ref{equ:weak formulation}, and using the operator $\mathcal{S}$, we equivalently rewrite problem (\ref{eqn:modified problems}) as the following reduced form :
\begin{equation}\label{eqn:reduced form}
  \left\{ \begin{aligned}
        &\min \limits_{u, z}^{}\ \ \hat{J}(u)+\delta_{U_{ad}}(z)\\
        &~~{\rm{s.t.}}\quad  u=z,
                          \end{aligned} \right.\tag{$\mathrm{RDP}$}
\end{equation}
with the reduced cost function
\begin{equation}\label{reduced cost function}
  \hat{J}(u):=J(\mathcal{S}(u+y_c),u)=\frac{1}{2}\|\mathcal{S}(u+y_c)-y_d\|_{L^2(\Omega)}^{2}+\frac{\alpha}{2}\|u\|_{L^2(\Omega)}^{2}.
\end{equation}
Let us define the augmented Lagrangian function of the problem (\ref{eqn:reduced form}) as follows:
\begin{equation}\label{augmented Lagrangian function}
  \mathcal{L}_\sigma(u,z;\lambda)=\hat{J}(u)+\delta_{U_{ad}}(z)+\langle\lambda,u-z\rangle_{L^2(\Omega)}+\frac{\sigma}{2}\|u-z\|_{L^2(\Omega)}^{2},
\end{equation}
with the Lagrange multiplier $\lambda \in L^2(\Omega)$ and $\sigma>0$ be a penalty parameter. Moreover, for the convergence property and the iteration complexity analysis, we define the function $R: (u,z,\lambda)\rightarrow [0,\infty)$ by:
\begin{equation}\label{KKT function}
  R(u,z,\lambda)=\|\nabla \hat{J}(u)+\lambda\|^2_{L^2(\Omega)}+{\rm dist}^2(0, -\lambda+\partial\delta_{U_{ad}}(z))+\|u-z\|^2_{L^2(\Omega)}.
\end{equation}

In addition, sometimes, it is expensive and unnecessary to exactly compute the solution of each subproblem even if it is doable, especially at the early stage of the whole process. For example, if a subproblem is equivalent to solving a large-scale or  ill-condition linear system, it is a natural idea to use the iterative methods such as some Krylov-based methods. Hence, taking the inexactness of the solution into account, a more practical inexact ADMM (iADMM) in function space is proposed for (\ref{eqn:reduced form}). The iterative scheme of inexact ADMM is shown in Algorithm \ref{algo1:ADMM for problems RP}.

\begin{algorithm}[H]\label{algo1:ADMM for problems RP}
  \caption{inexact ADMM algorithm for (\ref{eqn:reduced form})}
  \KwIn{$(z^0, u^0, \lambda^0)\in {\rm dom} (\delta_{U_{ad}}(\cdot))\times L^2(\Omega) \times L^2(\Omega)$ and a parameter $\tau \in (0,\frac{1+\sqrt{5}}{2})$. Let $\{\epsilon_k\}^\infty_{k=0}$ be a sequence satisfying $\{\epsilon_k\}^\infty_{k=0}\subseteq [0,+\infty)$ and $\sum\limits_{k=0}^{\infty}\epsilon_k<\infty$. Set $k=0$}
  \KwOut{$ u^k, z^{k}, \lambda^k$}
\begin{description}
\item[Step 1] Find an minizer (inexact)
\begin{equation*}
u^{k+1}=\arg\min \mathcal{L}_\sigma(u,z^k;\lambda^k)-\langle\delta^k, u\rangle
\end{equation*}
where the error vector ${\delta}^k$ satisfies $\|{\delta}^k\|_{L^2(\Omega)} \leq {\epsilon_k}$
\item[Step 2] Compute $z^{k+1}$ as follows:
\begin{equation*}
z^{k+1}=\arg\min\mathcal{L}_\sigma(u^{k+1},z;\lambda^k)
\end{equation*}
  \item[Step 3] Compute
  \begin{equation*}
       \lambda^{k+1} = \lambda^k+\tau\sigma(u^{k+1}-z^{k+1})
  \end{equation*}
  \item[Step 4] If a termination criterion is not met, set $k:=k+1$ and go to Step 1
\end{description}
\end{algorithm}

About the global convergence as well as the iteration complexity of the inexact ADMM for (\ref{eqn:modified problems}), we have the following results.
\begin{theorem}
Suppose that Assumption\ref{equ:assumption:1} holds. Let $(y^*,u^*,z^*,p^*,\lambda^*)$ is the KKT point of {\rm(\ref{eqn:modified problems})} which satisfies {\rm(\ref{eqn:KKT})}, then the sequence $\{(u^{k},z^{k},\lambda^k)\}$ is generated by Algorithm \ref{algo1:ADMM for problems RP} with the associated state $\{y^k\}$ and adjoint state $\{p^k\}$, then we have
\begin{eqnarray*}
  &&\lim\limits_{k\rightarrow\infty}^{}\{\|u^{k}-u^*\|_{L^{2}(\Omega)}+\|z^{k}-z^*\|_{L^{2}(\Omega)}+\|\lambda^{k}-\lambda^*\|_{L^{2}(\Omega)} \}= 0\\
   && \lim\limits_{k\rightarrow\infty}^{}\{\|y^{k}-y^*\|_{H_0^{1}(\Omega)}+\|p^{k}-p^*\|_{H_0^{1}(\Omega)} \}= 0
\end{eqnarray*}
Moreover, there exists a constant $C$ only depending on the initial point ${(u^0,z^0,\lambda^0)}$ and the optimal solution ${(u^*,z^*,\lambda^*)}$ such that for $k\geq1$,
\begin{eqnarray}
  &&\min\limits^{}_{1\leq i\leq k} \{R(u^i,z^i,\lambda^i)\}\leq\frac{C}{k}, \\
  &&\lim\limits^{}_{k\rightarrow\infty}\left(k\times\min\limits^{}_{1\leq i\leq k} \{R(u^i,z^i,\lambda^i)\}\right) =0.
\end{eqnarray}
where $R(\cdot)$ is defined as in {\rm(\ref{KKT function})}
\end{theorem}
\begin{proof}
The proof is a direct application of general inexact ADMM in Hilbert Space for the problem (\ref{eqn:reduced form}) and omitted here. We refer the reader to the literature \cite{SunToh1,inexactADMM}
\end{proof}

\begin{remark}
1). The first subproblems of Algorithm {\rm\ref{algo1:ADMM for problems RP}} is a convex differentiable optimization problem with respect to $u$, if we omit the error vector $\delta^k$, thus it is equivalent to solving the following system:
\begin{equation}\label{equ:first subproblem}
  \nabla_u\mathcal{L}_\sigma(u^{k+1},z^k;\lambda^k):=\mathcal{S}^{*}(\mathcal{S}(u^{k+1}+y_c)-y_d)+\alpha u^{k+1}+\lambda^k+\sigma(u^{k+1}-z^k)=0
\end{equation}
Since $y^{k+1}=\mathcal{S}(u^{k+1}+y_c)$, we define $p^{k+1}:=-\mathcal{S}^{*}(\mathcal{S}(u^{k+1}+y_c)-y_d)=\mathcal{S}^{*}(y_d-y)$, then we have
\begin{equation}\label{equ:first subproblem1}
(\alpha+\sigma)u^{k+1}-p^{k+1}+\lambda^k-\sigma z^k=0,
\end{equation}
namely, we should solve
\begin{equation}\label{equ:saddle point problems1}
\left[
  \begin{array}{ccc}
    I & 0 & \quad\mathcal{S}^{-*} \\
    0 & ({\alpha}+\sigma)I & \quad-I \\
    \mathcal{S}^{-1} & -I & \quad0 \\
  \end{array}
\right]\left[
         \begin{array}{c}
           y^{k+1} \\
           u^{k+1} \\
           p^{k+1} \\
         \end{array}
       \right]=\left[
                 \begin{array}{c}
                   y_d \\
                   \sigma z^k-\lambda^k \\
                   y_c \\
                 \end{array}
               \right],
\end{equation}
Moreover, 
we could eliminate the variable $p$ and derive the following reduced system{\rm:}
\begin{equation}\label{equ:saddle point problems2}
\left[
  \begin{array}{cc}
    ({\alpha}+\sigma)I & \quad \mathcal{S}^* \\
    -\mathcal{S} & \quad I\\
  \end{array}
\right]\left[
         \begin{array}{c}
           u^{k+1} \\
           y^{k+1} \\
         \end{array}
       \right]=\left[
                 \begin{array}{c}
                   \mathcal{S}^*y_d+\sigma z^k-\lambda^k \\
                   \mathcal{S}y_c,\\
                 \end{array}
               \right]
\end{equation}
where $I$ represents the identity operator. Clearly, linear system {\rm(\ref{equ:saddle point problems2})} can be formally regarded as a special case of the generalized saddle-point problem, however, in the numerical calculation, it is either impossible or extremely expensive to obtain exact solutions of {\rm(\ref{equ:saddle point problems2})}. This fact urges us to use an inexact versions of ADMM. This is also the reason that there is a error vector ${\delta}^k$ in the first subproblem of our algorithm. As we know, according to the structure of the linear system {\rm(\ref{equ:saddle point problems2})}, some \emph{Krylov}-based methods could be employed to inexactly solve the linear system by constructing a good preconditioning.

2). It is easy to see that $z$-subproblem has a closed solution:
\begin{equation}\label{z-closed form solution}
\begin{aligned}
z^{k+1}&=\Pi_{U_{ad}}(u^{k+1}+\frac{\lambda^k}{\sigma}),
\end{aligned}
\end{equation}
\end{remark}

Based on the well-formed structure of (\ref{equ:saddle point problems2}) and (\ref{z-closed form solution}), it will be a crucial point in the numerical analysis to establish relations parallel to
(\ref{equ:saddle point problems2}) and (\ref{z-closed form solution}) also for the discretized problem.

\subsection{Semismooth Newton methods in function space}
At the end of this section, let us introduce the primal and dual active set (PDAS) method as our Phase-II algorithm. As we know, with the solution of the Phase-I as a good initial point, the PDAS method could employ the second-order information to solve the discrete problem to high accuracy. For problem (\ref{eqn:orginal problems}), the unique optimal solution $(y^*,u^*)$ could be characterized by the following necessary and sufficient first-order optimality conditions.

\begin{equation}\label{eqn:KKT2}
 G(y^*,u^*,p^*,\mu)=\left( \begin{aligned}
      & \qquad \qquad \qquad y^*-\mathcal{S}(u^*+y_c)\\
       &\qquad \qquad \qquad p^*+\mathcal{S}^*(y^*-y_d)\\
       &\qquad \qquad\qquad \alpha u^*-p^*+\mu^*\\
       & \mu^*-\max(0, \mu^*+c(u^*-b))-\min(0, \mu^*+c(u^*-a))
                          \end{aligned} \right)=0
 \end{equation}
for any $c>0$.

Since equation (\ref{eqn:KKT2}) is not differentiable in the classical sense, thus a so-called semi-smooth Newton method can be applied and it is well posed.
%
%
And an semismooth Newton with active set strategy can be implemented. The full numerical scheme is summarized in Algorithm 1:

 \begin{algorithm}[H]
  \caption{\ Primal-Dual Active Set (PDAS) method for (\ref{eqn:orginal problems})}
  \label{algo1:Primal-Dual Active Set (PDAS) method}
Initialization: Choose $y^0$, $u^0$, $p^0$ and $\mu^0$. Set $k=0$ and $c>0$.
\begin{description}
\item[Step 1] Determine the following subsets of $\Omega$ (Active and Inactive sets)
\begin{eqnarray*}
&&\mathcal{A}^{k+1}_{a} = \{x\in \Omega:\mu^k(x)+c(u^k(x)-a)<0\}, \\
&&\mathcal{A}^{k+1}_{b}= \{x\in \Omega: \mu^k(x)+c(u^k(x)-b)>0\}, \\
&& \mathcal{I}^{k+1}= \Omega\backslash (\mathcal{A}^{k+1}_a \cup \mathcal{A}^{k+1}_a).
\end{eqnarray*}
\item[Step 2] solve the following system
\begin{eqnarray*}
\left\{\begin{aligned}
        &y^{k+1} - \mathcal{S}(u^{k+1}+y_c)=0, \\
  &p^{k+1} + \mathcal{S}^*(y^{k+1}-y_d)=0,\\
  &\alpha u^{k+1}-p^{k+1}+\mu^{k+1} = 0,
 \end{aligned}\right.
\end{eqnarray*}
 where
\begin{eqnarray*}
 &u^{k+1}=\left\{\begin{aligned}
  a\quad {\rm a.e.\ on}~ \mathcal{A}^{k+1}_{a}\\
  b\quad {\rm a.e.\ on}~ \mathcal{A}^{k+1}_{b}
   \end{aligned}\right. \qquad and\quad   \mu^{k+1}=0\quad {\rm on}~ \mathcal{I}^{k+1}
 \end{eqnarray*}
  \item[Step 3] If a termination criterion is not met, set $k:=k+1$ and go to Step 1
\end{description}
\end{algorithm}

The primal-dual active set strategy has been introduced in \cite{BeItKu} for control constrained elliptic optimal control problems. Its relation to semismooth Newton methods in $\mathbb{R}^n$ as well as in
function space as found in \cite{HiItKu} can be used to prove the following fast local convergence (see \cite{Ulbrich2, Ulbrich1} for more details).
\begin{theorem}\label{PDAS convergence}
Under Assumption {\rm \ref{equ:assumption:1}}, let $\{(y^k, u^k)\}$ be generated by Algorithm \ref{algo1:Primal-Dual Active Set (PDAS) method}. Then, if the initialization $(y^0, u^0)$ is sufficiently close to the solution $(y^*, u^*)$ of \ref{eqn:orginal problems}, the $\{(y^k, u^k)\}$ converge superlinearly to $(y^*, u^*)$ in $L^2(\Omega)\times H^1_0(\Omega)$.
\end{theorem}

Due to the efficient implementation of the two-phase framework with Algorithm \ref{algo1:ADMM for problems RP} in Phase-I and Algorithm \ref{algo1:Primal-Dual Active Set (PDAS) method} in Phase-II, it will be important to establish the extension of two-phase algorithm for the discretized problem.

\section{Finite Element Approximation}
\label{sec:3}
To numerically solve problem (\ref{eqn:orginal problems}), we consider the finite element method, in which the state $y$ and the control $u$ are both discretized by continuous piecewise linear functions.

To this aim, we first consider a family of regular and quasi-uniform triangulations $\{\mathcal{T}_h\}_{h>0}$ of $\bar{\Omega}$. For each cell $T\in \mathcal{T}_h$, let us define the diameter of the set $T$ by $\rho_{T}:={\rm diam}\ T$ and define $\sigma_{T}$ to be the diameter of the largest ball contained in $T$. The mesh size of the grid is defined by $h=\max_{T\in \mathcal{T}_h}\rho_{T}$. We suppose that the following regularity assumptions on the triangulation are satisfied which are standard in the context of error estimates.

\begin{assumption}[Regular and quasi-uniform triangulations]\label{regular and quasi-uniform triangulations}
There exist two positive constants $\kappa$ and $\tau$ such that
   \begin{equation*}
   \frac{\rho_{T}}{\sigma_{T}}\leq \kappa,\quad \frac{h}{\rho_{T}}\leq \tau,
 \end{equation*}
hold for all $T\in \mathcal{T}_h$ and all $h>0$. Moreover, let us define $\bar{\Omega}_h=\bigcup_{T\in \mathcal{T}_h}T$, and let ${\Omega}_h \subset\Omega$ and $\Gamma_h$ denote its interior and its boundary, respectively. In the case that $\Omega$ is a convex polyhedral domain, we have $\Omega=\Omega_h$. In case $\Omega$ with a $C^{1,1}$- boundary $\Gamma$, we assumed that $\bar{\Omega}_h$ is convex and that all boundary vertices of $\bar{\Omega}_h$ are contained in $\Gamma$, such that
\begin{equation*}
  |\Omega\backslash {\Omega}_h|\leq c h^2,
\end{equation*}
where $|\cdot|$ denotes the measure of the set and $c>0$ is a constant.
\end{assumption}
On account of the homogeneous boundary condition of the state equation, we use

\begin{eqnarray*}
  Y_h &=&\left\{y_h\in C(\bar{\Omega})~\big{|}~y_{h|T}\in \mathcal{P}_1~ {\rm{for\ all}}~ T\in \mathcal{T}_h~ \mathrm{and}~ y_h=0~ \mathrm{in } ~\bar{\Omega}\backslash {\Omega}_h\right\}
\end{eqnarray*}
as the discrete state space, where $\mathcal{P}_1$ denotes the space of polynomials of degree less than or equal to $1$. Next, we also consider use the same discrete space to discretize control $u$, thus we define
\begin{equation*}
   U_h =\left\{u_h\in C(\bar{\Omega})~\big{|}~u_{h|T}\in \mathcal{P}_1~ {\rm{for\ all}}~ T\in \mathcal{T}_h~ \mathrm{and}~ u_h=0~ \mathrm{in } ~\bar{\Omega}\backslash{\Omega}_h\right\}.
\end{equation*}
For a given regular and quasi-uniform triangulation $\mathcal{T}_h$ with nodes $\{x_i\}_{i=1}^{N_h}$, let $\{\phi_i(x)\} _{i=1}^{N_h}$ be a set of nodal basis functions associated with nodes $\{x_i\}_{i=1}^{N_h}$, which span $Y_h$ as well as $U_h$ and satisfy the following properties:
\begin{equation}\label{basic functions properties}
   \phi_i(x) \geq 0, \quad \|\phi_i(x)\|_{\infty} = 1 \quad \forall i=1,2,...,N_h,\quad \sum\limits_{i=1}^{N_h}\phi_i(x)=1.
\end{equation}
The elements $u_h\in U_h$ and $y_h\in Y_h$ can be represented in the following forms, respectively,
\begin{equation*}
  u_h=\sum \limits_{i=1}^{N_h}u_i\phi_i(x),\quad y_h=\sum \limits_{i=1}^{N_h}y_i\phi_i(x),
\end{equation*}
and $u_h(x_i)=u_i$ and $y_h(x_i)=y_i$ hold. Let $U_{ad,h}$ denotes the discrete feasible set, which is defined by
\begin{equation*}
  U_{ad,h}:=U_h\cap U_{ad}=\left\{z_h=\sum \limits_{i=1}^{N_h}z_i\phi_i(x)~\big{|}~a\leq z_i\leq b, \forall i=1,...,m\right\}\subset U_{ad}.
\end{equation*}
Now, we can consider a discrete version of the problem (\ref{eqn:orginal problems}) as:
\begin{equation}\label{eqn:discretized problems}
  \left\{ \begin{aligned}
        &\min \limits_{(y_h,u_h)\in Y_h\times U_h}^{}J_h(y_h,u_h)=\frac{1}{2}\|y_h-y_d\|_{L^2(\Omega_h)}^{2}+\frac{\alpha}{2}\|u_h\|_{L^2(\Omega_h)}^{2}\\
        &\qquad\quad {\rm{s.t.}}\qquad a(y_h, v_h)=\int_{\Omega}u_hv_h{\rm{d}}x \qquad  \forall v_h\in Y_h,  \\
          &\qquad \qquad \qquad~  u_h\in U_{ad,h}.
                          \end{aligned} \right.\tag{$\mathrm{P}_{h}$}
 \end{equation}
Followed the error estimates result in \cite{linearerror}, we have the following results.
\begin{theorem}\label{theorem:error1}
Let us assume that $u^*$ and $u^*_h$ be the optimal control solutions of {\rm(\ref{eqn:orginal problems})} and {\rm(\ref{eqn:discretized problems})}, respectively. Then the following identity holds
\begin{eqnarray*}
\lim\limits^{}_{h\rightarrow0}\frac{1}{h}\|u-u_h\|_{L^2(\Omega)}=0.
\end{eqnarray*}
\end{theorem}
Moreover, let
\begin{equation*}
   K_h =\left(a(\phi_i, \phi_j)\right)_{i,j=1}^n,\quad M_h=\left(\int_{\Omega_h}\phi_i\phi_j{\mathrm{d}}x\right)_{i,j=1}^n
\end{equation*}
be the finite element stiffness and mass matrixes, respectively.

Moreover, as a result of the requirement of following algorithms, we introduce the lump mass matrix $W_h$ which is a diagonal matrix as:
\begin{equation*}
  W_h={\rm{diag}}\left(\int_{\Omega_h}\phi_i(x)\mathrm{dx}\right)_{i,j=1}^n.
\end{equation*}
It should be mentioned that the lump mass matrix $W_h$ originates in the following nodal quadrature formulas to approximately discretized the \textsl{$L^2$}-norm:
\begin{equation*}
  \|z_h\|^2_{L^{2}_h(\Omega_h)}:=\sum\limits_{i=1}^{N_h}(z_i)^2\int_{\Omega_h}\phi_i(x)\mathrm{dx}=\|z\|^2_{W_h}.
\end{equation*}
and call them $L^{2}_h$-norm. It is obvious that the \textsl{$L^{2}_h$}-norm can be considered as a weighted $l^2$-norm of the coefficient of $z_h$. 
More importantly, we have the following results about the mass matrix $M_h$ and the lump mass matrix $W_h$.
\begin{proposition}{\rm\textbf{\cite[Table 1]{Wathen}}}\label{eqn:martix properties}
$\forall$ $z\in \mathbb{R}^{N_h}$, the following inequalities hold:
\begin{equation} \label{equ:martix properties1}
  \|z\|^2_{M_h}\leq\|z_h\|^2_{W_h}\leq c\|z_h\|^2_{M_h}, \quad where \quad c=
 \left\{ \begin{aligned}
         &4  \quad if \quad n=2, \\
         &5  \quad if \quad n=3.
                           \end{aligned} \right.                          \\
\end{equation}
\end{proposition}

Denoting by $y_{c,h}:=\sum\limits_{i=1}^{N_h}y_c^i\phi_i(x)$ and $y_{d,h}:=\sum\limits_{i=1}^{N_h}y_d^i\phi_i(x)$ the $L^2$-projection of $y_c$ and $y_d$ onto $Y_h$, respectively, where $y_d^i=y_d(x^i)$, and identifying discrete functions with their coefficient vectors, we can rewrite the problem (\ref{eqn:discretized problems}) as a matrix-vector form:
\begin{equation}\label{matrix vector form}
\left\{\begin{aligned}
        &\min\limits_{(y,u,z)\in\mathbb{R}^{3N_h}}^{}~~ \frac{1}{2}\|y-y_d\|_{M_h}^{2}+\frac{\alpha}{2}\|u\|_{M_h}^{2}\\
        &~~~ \quad {\rm{s.t.}}\qquad\quad K_hy=M_h(u+y_c),\\
        &\ \qquad\quad\quad\quad\quad u\in[a,b]^{N_h}.
                          \end{aligned} \right.\tag{$\overline{\mathrm{P}}_{h}$}
\end{equation}

\section{An inexact heterogeneous ADMM algorithm and two-phase strategy for discretized problems}\label{sec:4} 
In this section, we will introduce an inexact ADMM algorithm and a two-phase strategy for discrete problems. Firstly, in order to establish relations parallel to (\ref{equ:saddle point problems2}) and (\ref{z-closed form solution}) for the discrete problem (\ref{matrix vector form}), we propose an inexact heterogeneous ADMM (ihADMM) algorithm with the aim of solving (\ref{matrix vector form}) to moderate accuracy. Furthermore, as we have mentioned, if more accurate solutions is necessarily required, combining our ihADMM and the PDAS method is a wise choice. Then a two-phase strategy is introduced. Specifically, utilizing the solution generated by our ihADMM, as a reasonably good initial point, PDAS is used as a postprocessor of our ihADMM. 
Similar to continuous case, in order to get a separable form for problem (\ref{matrix vector form}), we introduce an artificial variable $z$ and equivalently rewrite the problem (\ref{matrix vector form}) as:

\begin{equation}\label{equ:seprable matrix-vector form}
\left\{\begin{aligned}
        &\min\limits_{y,u,z}^{}~~ \frac{1}{2}\|y-y_d\|_{M_h}^{2}+\frac{\alpha}{2}\|u\|_{M_h}^{2}+\delta_{[a,b]}(z)\\
        &\ {\rm{s.t.}}\quad K_hy=M_hu,\\
        &\ \quad\quad \  u=z.
                          \end{aligned} \right.\tag{$\overline{\mathrm{DP}}_{h}$}
\end{equation}
Since the stiffness matrix $K_h$ and the mass matrix $M_h$ are symmetric positive definite matrices, then problem (\ref{equ:seprable matrix-vector form}) can be rewritten the following reduced form:
\begin{equation}\label{equ:reduced seprable matrix-vector form}
\left\{\begin{aligned}
        &\min\limits_{u,z}^{}~~ f(u)+g(z)\\
        &\ {\rm{s.t.}}\quad u=z.
                          \end{aligned} \right.\tag{$\overline{\mathrm{RDP}}_{h}$}
\end{equation}
with the reduced cost functions
\begin{eqnarray}
 f(u)&:=& \frac{1}{2}\|K_h^{-1}M_h(u+y_c)-y_d\|_{M_h}^{2}+\frac{\alpha}{2}\|u\|_{M_h}^{2},\label{equ:fu function} \\
g(z) &:=&  \delta_{[a,b]^{N_h}}\label{equ:gz function}.
\end{eqnarray}

To solve (\ref{equ:reduced seprable matrix-vector form}) by using ADMM-type algorithm, we first introduce the augmented Lagrangian function for (\ref{equ:reduced seprable matrix-vector form}). According to three possible choices of norms ($\mathbb{R}^{N_h}$ norm, $W_h$-weighted norm and $M_h$-weighted norm), for the augmented Lagrangian function, there are three versions as follows: for given $\sigma>0$,
\begin{eqnarray}
\mathcal{L}^1_\sigma(u,z;\lambda)&:=&f(u)+g(z)+\langle\lambda,u-z\rangle+\frac{\sigma}{2}\|u-z\|^{2}, \label{aguLarg1}\\
\mathcal{L}^2_\sigma(u,z;\lambda)&:=&f(u)+g(z)+\langle\lambda,M_h(u-z)\rangle+\frac{\sigma}{2}\|u-z\|_{W_h}^{2}, \label{aguLarg2}\\
\mathcal{L}^3_\sigma(u,z;\lambda)&:=&f(u)+g(z)+\langle\lambda,M_h(u-z)\rangle+\frac{\sigma}{2}\|u-z\|_{M_h}^{2}\label{aguLarg3}.
\end{eqnarray}
Then based on these three versions of augmented Lagrangian function, we give the following four versions of ADMM-type algorithm for (\ref{equ:reduced seprable matrix-vector form}) at $k$-th ineration: for given $\tau>0$ and $\sigma>0$,

\begin{equation}\label{inexact ADMM1}
  \hspace{-1in}\left\{\begin{aligned}
  &u^{k+1}=\arg\min_u\ f(u)+\langle\lambda^k,u-z^k\rangle+\sigma/2\|u-z^k\|^{2},\\
  &z^{k+1}=\arg\min_z\ g(z)+\langle\lambda^k,u^{k+1}-z\rangle+\sigma/2\|u^{k+1}-z\|^{2},\\
  &\lambda^{k+1}=\lambda^k+\tau\sigma(u^{k+1}-z^{k+1}).
  \end{aligned}\right.\tag{ADMM1}
\end{equation}
\begin{equation}\label{inexact ADMM2}
\left\{\begin{aligned}
  &u^{k+1}=\arg\min_u\ f(u)+\langle\lambda^k,W_h(u-z^k)\rangle+\sigma/2\|u-z^k\|_{W_h}^{2},\\
  &z^{k+1}=\arg\min_z\ g(z)+\langle\lambda^k,W_h(u^{k+1}-z)\rangle+\sigma/2\|u^{k+1}-z\|_{W_h}^{2},\\
  &\lambda^{k+1}=\lambda^k+\tau\sigma(u^{k+1}-z^{k+1}).
  \end{aligned}\right.\tag{ADMM2}
\end{equation}
\begin{equation}\label{inexact ADMM3}
\left\{\begin{aligned}
  &u^{k+1}=\arg\min_u\ f(u)+\langle\lambda^k,M_h(u-z^k)\rangle+\sigma/2\|u-z^k\|_{M_h}^{2},\\
  &z^{k+1}=\arg\min_z\ g(z)+\langle\lambda^k,M_h(u^{k+1}-z)\rangle+\sigma/2\|u^{k+1}-z\|_{M_h}^{2},\\
  &\lambda^{k+1}=\lambda^k+\tau\sigma(u^{k+1}-z^{k+1}).
  \end{aligned}\right.\tag{ADMM3}
\end{equation}
\begin{equation}\label{inexact ADMM4}
\left\{\begin{aligned}
  &u^{k+1}=\arg\min_u\ f(u)+\langle\lambda^k,M_h(u-z^k)\rangle+{\color{blue}\sigma/2\|u-z^k\|_{M_h}^{2}},\\
  &z^{k+1}=\arg\min_z\ g(z)+\langle\lambda^k,M_h(u^{k+1}-z)\rangle+{\color{red}\sigma/2\|u^{k+1}-z\|_{W_h}^{2}},\\
  &\lambda^{k+1}=\lambda^k+\tau\sigma(u^{k+1}-z^{k+1}).
  \end{aligned}\right.\tag{ADMM4}
\end{equation}
As one may know, (\ref{inexact ADMM1}) is actually the classical ADMM for (\ref{equ:reduced seprable matrix-vector form}), meanwhile, (\ref{inexact ADMM2}), (\ref{inexact ADMM3}) and (\ref{inexact ADMM4}) can be regarded as three different discretized forms of Algorithm \ref{algo1:ADMM for problems RP}.

Now, let us start to analyze and compare the advantages and disadvantages of the four algorithms.
Firstly, we focus on the $z$-subproblem in each algorithm. Since both identity matrix $I$ and lumped mass matrix $W_h$ are diagonal, it is clear that all the $z$-subproblems in (\ref{inexact ADMM1}), (\ref{inexact ADMM2}) and (\ref{inexact ADMM4}) have a closed form solution, except for the $z$-subproblem in (\ref{inexact ADMM3}).
Specifically, for $z$-subproblem in (\ref{inexact ADMM1}), the closed form solution could be given by:
\begin{equation}\label{equ:closed form ADMM1}
z^k={\rm\Pi}_{U_{ad}}\left(u^{k+1}+\frac{\lambda^k}{\sigma}\right).
\end{equation}
Similarly, for $z$-subproblems in (\ref{inexact ADMM2}) and (\ref{inexact ADMM4}), the closed form solution could be given by:
\begin{equation}\label{equ:closed form ADMM24}
z^{k+1}={\rm\Pi}_{U_{ad}}\left( u^{k+1}+\frac{W_h^{-1}M_h\lambda^k}{\sigma}\right)
\end{equation}
Fortunately, the expressions of (\ref{equ:closed form ADMM1}) and (\ref{equ:closed form ADMM24}) are similar to (\ref{z-closed form solution}). As we have mentioned that, from the view of both the actual numerical implementation and convergence analysis of the algorithm, establishing such parallel relation is important.

In addition, to ensure the $z$-subproblem in (\ref{inexact ADMM3}) have a closed form solution, we could add a proximal term $\frac{\sigma}{2}\|z-z^{k}\|_{\theta I- M_h}^{2}$ to $z$-subproblem, thus we get
\begin{equation}\label{newsubproblem z}
\begin{aligned}
  z^{k+1}=&\arg\min_z\ \mathcal{L}^3_\sigma(u^{k+1},z;\lambda^{k})+\frac{\sigma}{2}\|z-z^{k}\|_{\theta I- M_h}^{2}\\
   =&{\rm\Pi}_{U_{ad}}\left(\frac{1}{\sigma\theta}M_h(\sigma u^{k+1}+\lambda^k-\sigma z^k)+z^k\right).
\end{aligned}
\end{equation}
In this case, we call this method the linearized ADMM (LADMM), one can refer to \cite{linearizedADMM}. It is well known that the efficiency of the LADMM depends on how close is the chosen $\theta$ to the minimal and optimal value $\|M_h\|_2$. However, this quantity is not always easily computable.

Next, let us analyze the structure of $u$-subproblem in each algorithm. For (\ref{inexact ADMM1}), the first subproblem at $k$-th iteration 
is equivalent to solving the following linear system:
\begin{equation}\label{eqn:saddle point1}
\left[
  \begin{array}{ccc}
    M_h & \quad0 & \quad K_h \\
    0 & \quad{\alpha}M_h+\sigma I & \quad-M_h \\
    K_h & \quad-M_h & \quad0 \\
  \end{array}
\right]\left[
         \begin{array}{c}
           y^{k+1} \\
           u^{k+1} \\
           p^{k+1} \\
         \end{array}
       \right]=\left[
                 \begin{array}{c}
                   M_hy_d \\
                   \sigma z^k-\lambda^k \\
                   M_hy_c \\
                 \end{array}
               \right].
\end{equation}
Similarly, the $u$-subproblem in (\ref{inexact ADMM2}) can be converted into the following linear system:
\begin{equation}\label{eqn:saddle point2}
\left[
  \begin{array}{ccc}
    M_h & \quad0 & \quad K_h \\
    0 & \quad{\alpha}M_h+\sigma W_h & \quad-M_h \\
    K_h & \quad-M_h & \quad0 \\
  \end{array}
\right]\left[
         \begin{array}{c}
           y^{k+1} \\
           u^{k+1} \\
           p^{k+1} \\
         \end{array}
       \right]=\left[
                 \begin{array}{c}
                   M_hy_d \\
                   \sigma W_h(z^k-\lambda^k) \\
                   M_hy_c \\
                 \end{array}
               \right].
\end{equation}

However, the $u$-subproblem in both (\ref{inexact ADMM3}) and (\ref{inexact ADMM4}) can be rewritten as:
\begin{equation}\label{eqn:saddle point3}
\left[
  \begin{array}{ccc}
    M_h & \quad0 & \quad K_h \\
    0 & \quad(\alpha+\sigma) M_h  & \quad-M_h \\
    K_h & \quad-M_h & \quad0 \\
  \end{array}
\right]\left[
         \begin{array}{c}
           y^{k+1} \\
           u^{k+1} \\
           p^{k+1} \\
         \end{array}
       \right]=\left[
                 \begin{array}{c}
                   M_hy_d \\
                   M_h(\sigma z^k-\lambda^k) \\
                   M_hy_c \\
                 \end{array}
               \right].
\end{equation}
In (\ref{eqn:saddle point3}), since $p^{k+1}=(\alpha+\sigma)u^{k+1}-\sigma z^k+\lambda^k$, it is obvious that (\ref{eqn:saddle point3}) can be reduced into the following system by eliminating the variable $p$ without any computational cost:
\begin{equation}\label{eqn:saddle point4}
\left[
  \begin{array}{cc}
    \frac{1}{\alpha+\sigma}M_h & K_h \\
    -K_h & M_h
  \end{array}
\right]\left[
         \begin{array}{c}
           y^{k+1} \\
           u^{k+1}
         \end{array}
       \right]=\left[
                 \begin{array}{c}
                   \frac{1}{\alpha+\sigma}(K_h(\sigma z^k-\lambda^k)+M_hy_d)\\
                   -M_hy_c
                 \end{array}
               \right],
\end{equation}
while, reduced forms of (\ref{eqn:saddle point1}) and (\ref{eqn:saddle point2}):
\begin{equation}\label{eqn:saddle point5}
\left[
  \begin{array}{cc}
    M_h & \frac{\alpha}{2}K_h+\sigma K_hM_h^{-1} \\
    -K_h & M_h
  \end{array}
\right]\left[
         \begin{array}{c}
           y^{k+1} \\
           u^{k+1}
         \end{array}
       \right]=\left[
                 \begin{array}{c}
                   (K_hM_h^{-1}(\sigma z^k-\lambda^k)+M_hy_d)\\
                   -M_hy_c
                 \end{array}
               \right],
\end{equation}
and
\begin{equation}\label{eqn:saddle point6}
\left[
  \begin{array}{cc}
    M_h & \frac{\alpha}{2}K_h+\sigma K_hM_h^{-1}W_h \\
    -K_h & M_h
  \end{array}
\right]\left[
         \begin{array}{c}
           y^{k+1} \\
           u^{k+1}
         \end{array}
       \right]=\left[
                 \begin{array}{c}
                   (K_hM_h^{-1}W_h(\sigma z^k-\lambda^k)+M_hy_d)\\
                   -M_hy_c
                 \end{array}
               \right],
\end{equation}
both involve the inversion of $M_h$.
For above mentioned reasons, we prefer to use (\ref{inexact ADMM4}), which is called the heterogeneous ADMM (hADMM). However, in general, it is expensive and unnecessary to exactly compute the solution of saddle point system (\ref{eqn:saddle point4}) even if it is doable, especially at the early stage of the whole process. Based on the structure of (\ref{eqn:saddle point4}), it is a natural idea to use the iterative methods such as some Krylov-based methods. Hence, taking the inexactness of the solution of $u$-subproblem into account, a more practical inexact heterogeneous ADMM (ihADMM) algorithm is proposed.

Due to the inexactness of the proposed algorithm, we first introduce an error tolerance. Throughout this paper, let $\{\epsilon_k\}$ be a summable sequence of nonnegative numbers, and define
\begin{equation}\label{error sequence}
  C_1:=\sum\limits^{\infty}_{k=0}\epsilon_k\leq\infty, \quad C_2:=\sum\limits^{\infty}_{k=0}\epsilon_k^2\leq\infty.
\end{equation}
The details of our ihADMM algorithm is shown in Algorithm \ref{algo4:inexact heterogeneous ADMM for problem RHP} to solve (\ref{equ:reduced seprable matrix-vector form}).

\begin{algorithm}[H]
  \caption{inexact heterogeneous ADMM algorithm for (\ref{equ:reduced seprable matrix-vector form})}\label{algo4:inexact heterogeneous ADMM for problem RHP}
  \textbf{Input}: {$(z^0, u^0, \lambda^0)\in {\rm dom} (\delta_{[a,b]}(\cdot))\times \mathbb{R}^n \times \mathbb{R}^n $ and parameters $\sigma>0$, $\tau>0$. 
  Set $k=1$.}\\
  \textbf{Output}: {$ u^k, z^{k}, \lambda^k$}
\begin{description}
\item[Step 1] Find an minizer (inexact)
\begin{eqnarray*}
  u^{k+1}&=&\arg\min f(u)+(M_h\lambda^k,u-z^k)
         +\frac{\sigma}{2}\|u-z^k\|_{M_h}^{2}-\langle\delta^k, u\rangle,
\end{eqnarray*}
where the error vector ${\delta}^k$ satisfies $\|{\delta}^k\|_{2} \leq {\epsilon_k}$
\item[Step 2] Compute $z^k$ as follows:
       \begin{eqnarray*}
       z^{k+1}&=&\arg\min g(z)+(M_h\lambda^k,u^{k+1}-z)
         +\frac{\sigma}{2}\|u^{k+1}-z\|_{W_h}^{2}
       \end{eqnarray*}
  \item[Step 3] Compute
  \begin{eqnarray*}
    \lambda^{k+1} &=& \lambda^k+\tau\sigma(u^{k+1}-z^{k+1}).
  \end{eqnarray*}

  \item[Step 4] If a termination criterion is not met, set $k:=k+1$ and go to Step 1
\end{description}
\end{algorithm}

\subsection{Convergence results of ihADMM}
For the ihADMM, in this section we establish the global convergence and the iteration complexity results in non-ergodic sense for the sequence generated by Algorithm \ref{algo4:inexact heterogeneous ADMM for problem RHP}.

Before giving the proof of Theorem \ref{discrete convergence results}, we first provide a lemma, which is useful for analyzing the non-ergodic iteration complexity of ihADMM and introduced in \cite{SunToh1}.
\begin{lemma}\label{complexity lemma}
If a sequence $\{a_i\}\in \mathbb{R}$ satisfies the following conditions:
\begin{eqnarray*}
  &&a_i\geq0 \ \text{for any}\ i\geq0\quad and\quad \sum\limits_{i=0}^{\infty}a_i=\bar a<\infty.
\end{eqnarray*}
Then we have
\begin{equation*}
  \min\limits^{}_{i=1,...,k}\{a_i\} \leq \frac{\bar a}{k}, \quad \lim\limits^{}_{k\rightarrow\infty} \{k\cdot\min\limits^{}_{i=1,...,k}\{a_i\}\} =0.
\end{equation*}
\end{lemma}

For the convenience of the iteration complexity analysis in below, we define the function $R_h: (u,z,\lambda)\rightarrow [0,\infty)$ by:
\begin{equation}\label{discrete KKT function}
  R_h(u,z,\lambda)=\|M_h\lambda+\nabla f(u)\|^2+{\rm dist}^2(0, -M_h\lambda+\partial g(z))+\|u-z\|^2.
\end{equation}

By the definitions of $f(u)$ and $g(z)$ in (\ref{equ:fu function}) and (\ref{equ:gz function}), it is obvious that $f(u)$ and $g(z)$ both are closed, proper and convex functions. Since $M_h$ and $K_h$ are symmetric positive definite matrixes, we know the gradient operator $\nabla f$ is strongly monotone, and we have
 \begin{equation}\label{subdifferential strongly monotone}
   \langle\nabla f(u_1)-\nabla f(u_2), u_1-u_2\rangle=\|u_1-u_2\|^2_{\Sigma_{f}}
   \end{equation}
where ${\Sigma_{f}}=\alpha M_h+M_hK_h^{-1}M_hK_h^{-1}M_h$ is also a symmetric positive definite matrix. Moreover, the subdifferential operator $\partial g$ is a maximal monotone operators, e.g.,
 \begin{equation}\label{subdifferential monotone}
   \langle\varphi_1-\varphi_2, z_1-z_2\rangle\geq0, \quad \forall\ \varphi_1\in\partial g(z_1),\ \varphi_2\in \partial g(z_2).
 \end{equation}
For the subsequent convergence analysis, we denote
\begin{eqnarray}
  \label{equ:exact u}\bar{u}^{k+1}&=&\arg\min \hat{J}_h(u)+\langle M_h\lambda^k,u-z^k\rangle
         +\frac{\sigma}{2}\|u-z^k\|_{M_h}^{2}\\
  \label{equ:exact z} \bar z^{k+1}&=&\Pi_{[a,b]}(\bar u^{k+1}+\frac{W_h^{-1}M_h\lambda^k}{\sigma})
\end{eqnarray}
which is the exact solutions at the $(k+1)$th iteration in Algorithm \ref{algo4:inexact heterogeneous ADMM for problem RHP}. The following result shows the gap between $(u^{k+1}, z^{k+1})$ and $(\bar u^{k+1}, \bar z^{k+1})$ in terms of the given error tolerance $\|{\delta}^k\|_{2} \leq {\epsilon_k}$.

\begin{lemma}\label{gap between exact and inexact solution}
  Let $\{(u^{k+1}, z^{k+1})\}$ be the squence generated by Algorithm {\rm\ref{algo4:inexact heterogeneous ADMM for problem RHP}}, and $\{\bar u^{k+1}\}$, $\{\bar z^{k+1}\}$ be defined by {\rm(\ref{equ:exact u})} and {\rm(\ref{equ:exact z})}. Then for any $k\geq0$, we have
  \begin{eqnarray}
    \label{equ:error u}\|u^{k+1}-\bar u^{k+1}\| &=&\|(\sigma M_h+\Sigma_f)^{-1}\delta^k\|\leq \rho\epsilon_k,  \\
    \label{equ:error z}\|z^{k+1}-\bar z^{k+1}\| &\leq&\|u^{k+1}-\bar u^{k+1}\|\leq \rho\epsilon_k,
  \end{eqnarray}
where $\rho:=\|(\sigma M_h+\Sigma_f)^{-1}\|$.
\end{lemma}
\begin{proof}
  By the optimality conditions at point $(u^{k+1}, z^{k+1})$ and $(\bar u^{k+1}, \bar z^{k+1})$, we have
  \begin{eqnarray*}
    &&\Sigma_f u^{k+1}-M_hK_h^{-1}M_h y_d +M_h\lambda^k+\sigma M_h(u^{k+1}-z^k)-\delta^k=0, \\
    &&\Sigma_f \bar u^{k+1}-M_hK_h^{-1}M_h y_d +M_h\lambda^k+\sigma M_h(\bar u^{k+1}-z^k)=0,
  \end{eqnarray*}
thus
\begin{eqnarray*}
   u^{k+1}-\bar u^{k+1}&=& (\sigma M_h+\Sigma_f)^{-1}\delta^k
\end{eqnarray*}
which implies (\ref{equ:error u}). From (\ref{equ:closed form ADMM24}) and (\ref{equ:exact z}), and the fact that the projection operator $\Pi_{[a,b]}(\cdot)$ is nonexpansive, we get
\begin{equation*}
  \|z^{k+1}-\bar z^{k+1}\|=\|\Pi_{[a,b]}( u^{k+1}+\frac{W_h^{-1}M_h\lambda^k}{\sigma})-\Pi_{[a,b]}(\bar u^{k+1}+\frac{W_h^{-1}M_h\lambda^k}{\sigma})\|
  \leq\|u^{k+1}-\bar u^{k+1}\|.
\end{equation*}
The proof is completed.
\end{proof}

Next, for $k\geq0$, we define
\begin{eqnarray*}
  &r^k=u^k-z^k,\quad \bar r^k=\bar u^k-\bar z^k&  \\
  &\tilde{\lambda }^{k+1}=\lambda^k+\sigma r^{k+1},\quad \bar{\lambda }^{k+1}=\lambda^k+\tau\sigma \bar r^{k+1}, \quad \hat{\lambda }^{k+1}=\lambda^k+\sigma \bar r^{k+1},&
\end{eqnarray*}
and give two inequalities which is essential for establishing both the global convergence and the iteration complexity of our ihADMM. For the details of the proof, one can see in Appendix.

\begin{proposition}\label{descent proposition}
Let $\{(u^{k}, z^{k}, \lambda^{k})\}$ be the sequence generated by Algorithm {\rm\ref{algo4:inexact heterogeneous ADMM for problem RHP}} and $(u^{*}, z^{*}, \lambda^{*})$ be the KKT point of problem {\rm(\ref{equ:reduced seprable matrix-vector form})}. Then for $k\geq0$ we have
\begin{equation}\label{inequlaities property1}
  \begin{aligned}
  &\langle\delta^k,u^{k+1}-u^*\rangle +\frac{1}{2\tau\sigma}\|\lambda^k-\lambda^*\|^2_{M_h}+\frac{\sigma}{2}\|z^k-z^*\|^2_{M_h}
  -\frac{1}{2\tau\sigma}\|\lambda^{k+1}-\lambda^*\|^2_{M_h}-\frac{\sigma}{2}\|z^{k+1}-z^*\|^2_{M_h}\\
  &\geq\|u^{k+1}-u^*\|^2_{T}
  +\frac{\sigma}{2}\|z^{k+1}-z^*\|^2_{W_h-M_h} +\frac{\sigma}{2}\|r^{k+1}\|^2_{W_h-\tau M_h}
  +\frac{\sigma}{2}\|u^{k+1}-z^k\|^2_{M_h},
  \end{aligned}
\end{equation}
where $T:=\Sigma_f-\frac{\sigma}{2}(W_h-M_h)$.
\end{proposition}

\begin{proposition}\label{descent proposition2}
Let $\{(u^{k}, z^{k}, \lambda^{k})\}$ be the sequence generated by Algorithm {\rm\ref{algo4:inexact heterogeneous ADMM for problem RHP}}, $(u^{*}, z^{*}, \lambda^{*})$ be the KKT point of the problem {\rm(\ref{equ:reduced seprable matrix-vector form})} and $\{\bar u^k\}$ and $\{\bar z^k\}$ be two sequences defined in {\rm(\ref{equ:exact u})} and {\rm(\ref{equ:exact z})}, respectively. Then for $k\geq0$ we have
\begin{equation}\label{inequlaities property}
  \begin{aligned}
 &\frac{1}{2\tau\sigma}\|\lambda^k-\lambda^*\|^2_{M_h}+\frac{\sigma}{2}\|z^k-z^*\|^2_{M_h}
  -\frac{1}{2\tau\sigma}\|\bar \lambda^{k+1}-\lambda^*\|^2_{M_h}-\frac{\sigma}{2}\|\bar z^{k+1}-z^*\|^2_{M_h}\\
  \geq\ &\|\bar u^{k+1}-u^*\|^2_{T}
  +\frac{\sigma}{2}\|\bar z^{k+1}-z^*\|^2_{W_h-M_h} +\frac{\sigma}{2}\|\bar r^{k+1}\|^2_{W_h-\tau M_h}+\frac{\sigma}{2}\|\bar u^{k+1}-z^k\|^2_{M_h},
  \end{aligned}
\end{equation}
where $T:=\Sigma_f-\frac{\sigma}{2}(W_h-M_h)$.
\end{proposition}
Then based on former results, we have the following convergence results.
\begin{theorem}\label{discrete convergence results}
Let $(y^*,u^*,z^*,p^*,\lambda^*)$ is the KKT point of {\rm(\ref{equ:seprable matrix-vector form})}, then the sequence $\{(u^{k},z^{k},\lambda^k)\}$ is generated by Algorithm {\rm\ref{algo4:inexact heterogeneous ADMM for problem RHP}} with the associated state $\{y^k\}$ and adjoint state $\{p^k\}$, then for any $\tau\in (0,1]$ and $\sigma\in (0, \frac{1}{4}\alpha]$, we have
\begin{eqnarray}
  \label{discrete iteration squence convergence1}&&\lim\limits_{k\rightarrow\infty}^{}\{\|u^{k}-u^*\|+\|z^{k}-z^*\|+\|\lambda^{k}-\lambda^*\| \}= 0\\
   \label{discrete iteration squence convergence2}&& \lim\limits_{k\rightarrow\infty}^{}\{\|y^{k}-y^*\|+\|p^{k}-p^*\| \}= 0
\end{eqnarray}
Moreover, there exists a constant $C$ only depending on the initial point ${(u^0,z^0,\lambda^0)}$ and the optimal solution ${(u^*,z^*,\lambda^*)}$ such that for $k\geq1$,
\begin{eqnarray}
  \label{discrete iteration complexity1}&&\min\limits^{}_{1\leq i\leq k} \{R_h(u^i,z^i,\lambda^i)\}\leq\frac{C}{k}, \quad
\lim\limits^{}_{k\rightarrow\infty}\left(k\times\min\limits^{}_{1\leq i\leq k} \{R_h(u^i,z^i,\lambda^i)\}\right) =0.
\end{eqnarray}
where $R_h(\cdot)$ is defined as in {\rm(\ref{discrete KKT function})}.

\begin{proof}
It is easy to see that $(u^*,z^*)$ is the unique optimal solution of discrete problem (\ref{equ:reduced seprable matrix-vector form}) if and only if there exists a Lagrangian multiplier $\lambda^*$ such that the following Karush-Kuhn-Tucker (KKT) conditions hold,
\begin{subequations}
\begin{eqnarray}
  \label{equ: exact variational inequalities1}&-M_h\lambda^*=\nabla f(u^*),\\
  \label{equ: exact variational inequalities2}&M_h\lambda^*\in \partial g(z^*),\\
  \label{equ: exact variational inequalities3}& u^*=z^*.
\end{eqnarray}
\end{subequations}

In the inexact heterogeneous ADMM iteration scheme, the optimality conditions for $(u^{k+1}, z^{k+1})$ are
\begin{subequations}
\begin{eqnarray}
  \label{equ: inexact variational inequalities1}&\delta^k-(M_h\lambda^k+\sigma M_h(u^{k+1}-z^k))=\nabla f(u^{k+1}),\\
  \label{equ: inexact variational inequalities2}&M_h\lambda^k+\sigma W_h(u^{k+1}-z^{k+1})\in \partial g(z^{k+1}).
\end{eqnarray}
\end{subequations}

Next, let us first prove the \textbf{global convergence of iteration sequences,} e.g., establish the proof of (\ref{discrete iteration squence convergence1}) and (\ref{discrete iteration squence convergence2}).
The first step is to show that $\{(u^k, z^k, \lambda^k)\}$ is bounded. We define the following sequence $\theta^k$ and $\bar\theta^k$ with:
\begin{eqnarray}
\theta^k &=& \left(\frac{1}{\sqrt{2\tau\sigma}}M_h^{\frac{1}{2}}(\lambda^k-\lambda^*), \sqrt{\frac{\sigma}{2}}M_h^{\frac{1}{2}}(z^k-z^*)\right),\label{iteration sequence1} \\
   \bar\theta^k &=& \left(\frac{1}{\sqrt{2\tau\sigma}}M_h^{\frac{1}{2}}(\bar\lambda^k-\lambda^*), \sqrt{\frac{\sigma}{2}}M_h^{\frac{1}{2}}(\bar z^k-z^*)\right)\label{iteration sequence2}.
\end{eqnarray}
According to Proposition \ref{eqn:martix properties}, for any $\tau\in (0,1]$ and $\sigma\in (0, \frac{1}{4}\alpha]$, we have
\begin{equation}\label{positive definite matrix}
  \begin{aligned}
  \Sigma_f-\frac{\sigma}{2}(W_h-M_h) \succ 0, \quad W_h-\tau M_h \succ 0 .
  \end{aligned}
\end{equation}
Then, by Proposition \ref{descent proposition2}, we get $\|\bar\theta^{k+1}\|^2\leq\|\theta^k\|^2$. As a result, we have:
\begin{equation}\label{descent theta}
  \begin{aligned}
  \|\theta^{k+1}\| &\leq \|\bar\theta^{k+1}\|+\|\bar\theta^{k+1}-\theta^{k+1}\|= \|\theta^{k}\|+\|\bar\theta^{k+1}-\theta^{k+1}\| .
  \end{aligned}
\end{equation}
Employing Lemma \ref{gap between exact and inexact solution}, we get
\begin{equation}\label{descent theta and bartheta}
  \begin{aligned}
  \|\bar\theta^{k+1}-\theta^{k+1}\|^2 &= \frac{1}{2\tau\sigma}\|\bar\lambda^{k+1}-\lambda^{k+1}\|^2_{M_h}+\frac{\sigma}{2}\|\bar z^{k+1}-z^{k+1}\|^2_{M_h} \\
  &= \frac{\tau\sigma}{2}\|\bar r^{k+1}-r^{k+1}\|^2_{M_h}+\frac{\sigma}{2}\|\bar z^{k+1}-z^{k+1}\|^2_{M_h}  \\
  &\leq \tau\sigma \|\bar u^{k+1}-u^{k+1}\|^2_{M_h}+(\tau\sigma+\frac{\sigma}{2})\|\bar z^{k+1}-z^{k+1}\|^2_{M_h}\\
  &\leq(2\tau+1/2)\sigma\|M_h\|\rho^2\epsilon_k^2\leq5/2\sigma\|M_h\|\rho^2\epsilon_k^2,
  \end{aligned}
\end{equation}
which implies $\|\bar\theta^{k+1}-\theta^{k+1}\|\leq\sqrt{5/2\sigma\|M_h\|}\rho\epsilon_k$. Hence, for any $k\geq0$, we have
\begin{equation}\label{boundedness theta and bartheta}
  \begin{aligned}
  \|\theta^{k+1}\| &\leq \|\theta^k\|+\sqrt{5/2\sigma\|M_h\|}\rho\epsilon_k\\
  &\leq  \|\theta^0\|+\sqrt{5/2\sigma\|M_h\|}\rho\sum\limits^{\infty}_{k=0}\epsilon_k=\|\theta^0\|+\sqrt{5/2\sigma\|M_h\|}\rho C_1\equiv\bar\rho.
  \end{aligned}
\end{equation}
From $\|\bar\theta^{k+1}\|\leq\|\theta^{k}\|$, for any $k\geq0$, we also have $\|\bar\theta^{k+1}\|\leq\bar\rho$. Therefore, the sequences $\{\theta^k\}$ and $\{\bar \theta^k\}$ are bounded. From the definition of $\{\theta^k\}$ and the fact that $M_h\succ0$, we can see that the sequences $\{\lambda^k\}$ and $\{z^k\}$ are bounded. Moreover, from updating technique of $\lambda^k$, we know $\{u^k\}$ is also bounded. Thus, due to the boundedness of the sequence $\{(u^{k}, z^{k}, \lambda^k)\}$, we know the sequence has a subsequence
$\{(u^{k_i}, z^{k_i}, \lambda^{k_i}))\}$ which converges to an accumulation point $(\bar u, \bar z, \bar\lambda)$. Next we should show that $(\bar u, \bar z, \bar\lambda)$ is a KKT point and equal to $(u^*, z^*, \lambda^*)$. 

Again employing Proposition \ref{descent proposition2}, we can derive
\begin{equation}\label{convergence inequlity}
  \begin{aligned}
&\sum\limits^{\infty}_{k=0}\left(\|\bar u^{k+1}-u^*\|^2_{T}
  +\frac{\sigma}{2}\|\bar z^{k+1}-z^*\|^2_{W_h-M_h} +\frac{\sigma}{2}\|\bar r^{k+1}\|^2_{W_h-\tau M_h}+\frac{\sigma}{2}\|\bar u^{k+1}-z^k\|^2_{M_h}\right)\\
  \leq &\sum\limits^{\infty}_{k=0}(\|\theta^k\|^2-\|\theta^{k+1}\|^2+\|\theta^{k+1}\|^2-\|\bar \theta^{k+1}\|^2) \\
  \leq&\|\theta^0\|^2+\sum\limits^{\infty}_{k=0}\left(\|\theta^{k+1}-\bar\theta^{k+1}\|(\|\theta^{k+1}\|+\|\bar\theta^{k+1}\|)\right) \\
  \leq& \|\theta^0\|^2+2\bar\rho\sqrt{5/2\sigma\|M_h\|}\rho C_1<\infty
  \end{aligned}
\end{equation}
which means
\begin{equation}\label{limit convergence1}
  \begin{aligned}
  \lim\limits^{}_{k\rightarrow\infty}\|\bar u^{k+1}-u^*\|_{T}=0,&\quad \lim\limits^{}_{k\rightarrow\infty}\|\bar z^{k+1}-z^*\|_{W_h-M_h}=0,&\\
  \lim\limits^{}_{k\rightarrow\infty} \|\bar r^{k+1}\|_{W_h-\tau M_h}=0,&\quad
  \lim\limits^{}_{k\rightarrow\infty}\|\bar u^{k+1}-z^k\|_{M_h}=0&.
    \end{aligned}
\end{equation}
Note that $T\succ0, W_h-M_h\succ0, W_h-\tau M_h\succ0$ and $M_h\succ0$, then we have
\begin{equation}\label{limit convergence2}
  \begin{aligned}
  \lim\limits^{}_{k\rightarrow\infty}\|\bar u^{k+1}-u^*\|=0,&\quad \lim\limits^{}_{k\rightarrow\infty}\|\bar z^{k+1}-z^*\|=0,&\\
  \lim\limits^{}_{k\rightarrow\infty} \|\bar r^{k+1}\|=0,&\quad
  \lim\limits^{}_{k\rightarrow\infty}\|\bar u^{k+1}-z^k\|=0.&
    \end{aligned}
\end{equation}
From the Lemma \ref{gap between exact and inexact solution}, we can get
\begin{equation}\label{limit convergence3}
  \begin{aligned}
  &\|u^{k+1}-u^*\|\leq\|\bar u^{k+1}-u^*\|+\|u^{k+1}-\bar u^{k+1}\|\leq\|\bar u^{k+1}-u^*\|+\rho\epsilon_k,\\
  &\|z^{k+1}-z^*\|\leq\|\bar z^{k+1}-z^*\|+\|z^{k+1}-\bar z^{k+1}\|\leq\|\bar z^{k+1}-z^*\|+\rho\epsilon_k.
    \end{aligned}
\end{equation}
From the fact that $\lim\limits^{}_{k\rightarrow\infty} \epsilon_k=0$ and (\ref{limit convergence2}), by taking the limit of both sides of (\ref{limit convergence3}), we have
\begin{equation}\label{limit convergence4}
  \begin{aligned}
  \lim\limits^{}_{k\rightarrow\infty}\| u^{k+1}-u^*\|=0,&\quad \lim\limits^{}_{k\rightarrow\infty}\| z^{k+1}-z^*\|=0,&\\
  \lim\limits^{}_{k\rightarrow\infty} \| r^{k+1}\|=0,&\quad
  \lim\limits^{}_{k\rightarrow\infty}\| u^{k+1}-z^k\|=0.&
    \end{aligned}
\end{equation}
Now taking limits for $k_i\rightarrow\infty$ on both sides of (\ref{equ: inexact variational inequalities1}), we have
\begin{equation*}
  \lim\limits^{}_{k_i\rightarrow\infty}(\delta^{k_i}-(M_h\lambda^{k_i}+\sigma M_h(u^{k_i+1}-z^{k_i})))=\nabla f(u^{k_i+1}),
\end{equation*}
which results in
\begin{equation*}
  -M_h\bar\lambda=\nabla f(u^*)
\end{equation*}
Then from (\ref{equ: exact variational inequalities1}), we know $\bar\lambda=\lambda^*$. At last, to complete the proof, we need to show that $\lambda^*$ is the limit of the sequence of $\{\lambda^k\}$. From
(\ref{boundedness theta and bartheta}), we have for any $k>k_i$,
\begin{equation*}
  \|\theta^{k+1}\|\leq\|\theta^{k_i}\|+\sqrt{5/2\sigma\|M_h\|}\rho\sum\limits^{k}_{j={k_i}}\epsilon_j.
\end{equation*}
Since $\lim\limits^{}_{k_i\rightarrow\infty}\|\theta^{k_i}\|=0$ and $\sum\limits_{k=0}^{\infty}\epsilon_k<\infty$, we have that $\lim\limits^{}_{k\rightarrow\infty}\|\theta^{k}\|=0$, which implies
\begin{equation*}
  \lim\limits^{}_{k\rightarrow\infty}\| \lambda^{k+1}-\lambda^*\|=0.
\end{equation*}
Hence, we have proved the convergence of the sequence $\{(u^{k+1}, z^{k+1}, \lambda^{k+1})\}$, which completes the proof of (\ref{discrete iteration squence convergence1}) in Theorem \ref{discrete convergence results}. For the proof of (\ref{discrete iteration squence convergence2}), it is easily to show by the definition of the sequence $\{(y^k, p^k)\}$, here we omit it.

At last, we establish the proof of (\ref{discrete iteration complexity1}), e.g., \textbf{the iteration complexity results in non-ergodic sendse for the sequence generated by the ihADMM.}

Firstly, by the optimality condition (\ref{equ: inexact variational inequalities1}) and (\ref{equ: inexact variational inequalities2}) for $(u^{k+1}, z^{k+1})$, we have
\begin{subequations}
\begin{eqnarray}
 \label{equ:discrete inexact variational inequalities3}&\delta^k+(\tau-1)\sigma M_h r^{k+1}-\sigma M_h(z^{k+1}-z^k)=M_h\lambda^{k+1}+\nabla f(u^{k+1}),\\
  \label{equ:discrete inexact variational inequalities4}&\sigma (W_h-\tau M_h)r^{k+1}\in -M_h\lambda^{k+1}+\partial g(z^{k+1}).
\end{eqnarray}
\end{subequations}
By the definition of $R_h$ and denoting $w^{k+1}:=(u^{k+1},z^{k+1},\lambda^{k+1})$, we derive
\begin{equation}\label{discrete KKT function for k+1}\small
\begin{aligned}
  R_h(w^{k+1})&=\|M_h\lambda^{k+1}+\nabla f(u^{k+1})\|^2+{\rm dist}^2(0, -M_h\lambda^{k+1}+\partial g(z^{k+1}))+\|u^{k+1}-z^{k+1}\|^2\\
  &\leq 2\|\delta^k\|^2+2(\tau-1)^2\sigma^2 \|M_h\|^2\| r^{k+1}\|^2+2\sigma^2 \|M_h\|\|z^{k+1}-z^k\|_{M_h}^2\\
  &\quad +\sigma^2\| (W_h-\tau M_h)\|^2\|r^{k+1}\|^2+\|r^{k+1}\|^2\\
  &\leq 2\|\delta^k\|^2+2(\tau-1)^2\sigma^2 \|M_h\|^2\| r^{k+1}\|^2+4\sigma^2 \|M_h\|^2\|u^{k+1}-z^{k+1}\|^2\\
  &\quad +4\sigma^2 \|M_h\|\|u^{k+1}-z^{k}\|_{M_h}^2+\sigma^2\| W_h-\tau M_h\|^2\|r^{k+1}\|^2+\|r^{k+1}\|^2\\
  &=2\|\delta^k\|^2+\eta\|r^{k+1}\|^2+2\sigma^2 \|M_h\|\|u^{k+1}-z^{k}\|_{M_h}^2,
\end{aligned}
\end{equation}
where
\begin{equation*}
  \eta:=2(\tau-1)^2\sigma^2 \|M_h\|^2+2\sigma^2 \|M_h\|^2+\sigma^2\| W_h-\tau M_h\|^2+1.
\end{equation*}
In order to get a upper bound for $R_h(w^{k+1})$, we will use (\ref{inequlaities property1}) in Proposition
\ref{descent proposition}. First, by the definition of $\theta^{k}$ and (\ref{boundedness theta and bartheta}), for any $k\geq0$ we can easily have
\begin{equation*}
  \|\lambda^k-\lambda^*\| \leq\bar\rho\sqrt{\frac{2\tau\sigma}{\|M_h^{-1}\|}} , \quad
  \|z^k-z^*\| \leq\bar\rho\sqrt{\frac{2}{\sigma\|M_h^{-1}\|}}.
\end{equation*}
Next, we should give a upper bound for $\langle\delta^k, u^{k+1}-u^*\rangle$:
\begin{equation}\label{inner product estimates}
\begin{aligned}
\langle\delta^k, u^{k+1}-u^*\rangle&\leq\|\delta^k\|(\|u^{k+1}-z^{k+1}\|+\|z^{k+1}-z^*\|)\\
              &=\|\delta^k\|\left(\frac{1}{\tau\sigma}\|\lambda^{k+1}-\lambda^{k}\|+\|z^{k+1}-z^*\|\right)\\
              &\leq\left(\left(1+\frac{2}{\sqrt{\tau}}\right)\frac{2\sqrt{2}\bar\rho}{\sqrt{\tau\sigma\|M_h^{-1}\|}}\right)\|\delta^k\|\equiv \bar\eta\|\delta^k\|.
\end{aligned}
\end{equation}
Then by the (\ref{inequlaities property1}) in Proposition \ref{descent proposition}, we have
\begin{equation}\label{summable1}\small
  \begin{aligned}
  \sum\limits^{\infty}_{k=0}\left(\frac{\sigma}{2}\|r^{k+1}\|^2_{W_h-\tau M_h}
  +\frac{\sigma}{2}\|u^{k+1}-z^k\|^2_{M_h}\right)&\leq \sum\limits^{\infty}_{k=0}(\theta^{k}-\theta^{k+1})
  +\sum\limits^{\infty}_{k=0}\langle\delta^k,u^{k+1}-u^*\rangle \\
  &\leq \theta^0+\bar\eta\sum\limits^{\infty}_{k=0}\|\delta^k\|
  \leq\theta^0+\bar\eta\sum\limits^{\infty}_{k=0}\epsilon^k=\theta^0+\bar\eta C_1.
  \end{aligned}
\end{equation}
Hence,
\begin{equation}\label{summable2}
  \begin{aligned}
\sum\limits^{\infty}_{k=0}\|r^{k+1}\|^2\leq \frac{2(\theta^0+\bar\eta C_1)}{\sigma\|(W_h-\tau M_h)^{-1}\|}, \quad
\sum\limits^{\infty}_{k=0}\|u^{k+1}-z^k\|^2_{M_h}\leq \frac{2(\theta^0+\bar\eta C_1)}{\sigma}.
  \end{aligned}
\end{equation}
By substituting (\ref{summable2}) to (\ref{discrete KKT function for k+1}), we have
\begin{equation}\label{summable3}
  \begin{aligned}
\sum\limits^{\infty}_{k=0}R_h(w^{k+1})&\leq2\sum\limits^{\infty}_{k=0}\|\delta\|^2
+\eta\sum\limits^{\infty}_{k=0}\|r^{k+1}\|^2+2\sigma^2 \|M_h\|\sum\limits^{\infty}_{k=0}\|u^{k+1}-z^{k}\|_{M_h}^2\\
&\leq2\sum\limits^{\infty}_{k=0}\epsilon_k^2+ \eta\frac{2(\theta^0+\bar\eta C_1)}{\sigma\|(W_h-\tau M_h)^{-1}\|}+2\sigma^2 \|M_h\|\frac{2(\theta^0+\bar\eta C_1)}{\sigma}\\
&\leq C:=2C_2+\eta\frac{2(\theta^0+\bar\eta C_1)}{\sigma\|(W_h-\tau M_h)^{-1}\|}+2\sigma^2 \|M_h\|\frac{2(\theta^0+\bar\eta C_1)}{\sigma}
  \end{aligned}
\end{equation}
Thus, by Lemma \ref{complexity lemma}, we know (\ref{discrete iteration complexity1}) holds. Therefore, combining the obtained global convergence results, we complete the whole proof of the Theorem \ref{discrete convergence results}.
\end{proof}
\end{theorem}

\subsection{Numerical computation of the $u$-subproblem of Algorithm \ref{algo4:inexact heterogeneous ADMM for problem RHP}}\label{subsection linear sysytem}
\subsubsection{Error analysis of the linear system {(\ref{eqn:saddle point4})}}
As we know, the linear system (\ref{eqn:saddle point4}) is a special case of the generalized saddle-point problem, thus some Krylov-based methods could be employed to inexactly solve the linear system. Let $(r^k_1, r^k_2)$ be the residual error vector, that means:
\begin{equation}\label{eqn:saddle point4 with error}
\left[
  \begin{array}{cc}
    \frac{1}{\alpha+\sigma}M_h & K_h \\
    -K_h & M_h
  \end{array}
\right]\left[
         \begin{array}{c}
           y^{k+1} \\
           u^{k+1}
         \end{array}
       \right]=\left[
                 \begin{array}{c}
                   \frac{1}{\alpha+\sigma}(K_h(\sigma z^k-\lambda^k)+M_hy_d)+r_1\\
                   -M_hy_c+r_2
                 \end{array}
               \right],
\end{equation}
and $\delta^k=(\alpha+\sigma)M_hK_h^{-1}r_1^k+M_hK_h^{-1}M_hK_h^{-1}r_2^k$, thus in the numerical implementation we require
\begin{equation}\label{error estimates1}
  \|r^k_1\|_{2}+\|r^k_2\|_{2}\leq\frac{\epsilon_k}{\sqrt{2}\|M_hK_h^{-1}\|_2\max\{\|M_hK_h^{-1}\|_{2},\alpha+\sigma\}}
\end{equation}
to guarantee the error vector $\|{\delta}^k\|_{2} \leq {\epsilon_k}$.

\subsubsection{An efficient precondition techniques for solving the linear systems}\label{precondition}
To solve (\ref{eqn:saddle point4}), in this paper, we use the generalized minimal residual {\rm (GMRES)} method. In order to speed up the convergence of the {\rm GMRES} method, the preconditioned variant of modified hermitian and skew-hermitian splitting {\rm(PMHSS)} preconditioner $\mathcal{P}$ is employed which is introduced in {\rm \cite{Bai}}:
\begin{equation}\label{precondition matrix1}
  \mathcal{P_{HSS}}=\frac{1}{\gamma}\left[
    \begin{array}{lcc}
I & \quad\sqrt{\gamma}I \\
      -\sqrt{\gamma}I & \quad\gamma I \\
    \end{array}
  \right]\left[
           \begin{array}{lcc}
             M_h+\sqrt{\gamma}K_h & 0 \\
             0 & M_h+\sqrt{\gamma}K_h \\
           \end{array}
         \right],
\end{equation}
where $\gamma=\alpha+\sigma$. Let $\mathcal{A}$ denote the coefficient matrix of linear system (\ref{eqn:saddle point4}).
%

In actual implementations, the action of the preconditioning matrix, when used to
precondition the {\rm GMRES} methods, is realized through solving a sequence of generalized residual equations of the form $\mathcal{P_{HSS}}v=r$, where $r=(r_a; r_b)\in \mathbb{R}^{2N_h}$, with $r_a, r_b\in \mathbb{R}^{N_h}$. By making using of the structure of the matrix $\mathcal{P_{HSS}}$, we obtain the following procedure for computing the vector $v$:
  \begin{itemize}
  \item []Step 1 compute $\hat{r}_a$ and $\hat{r}_b$
  \begin{eqnarray*}
   \hat{r}_a&=&\frac{\gamma}{2}r_a-\frac{\sqrt{\gamma}}{2}r_b, \\
   \hat{r}_b&=&\frac{\sqrt{\gamma}}{2}r_a+\frac{1}{2}r_b.
  \end{eqnarray*}
  \item []Step 2 compute $v_a$ and $v_b$ by solving the following linear systerms
    \begin{eqnarray*}
   (M_h+\sqrt{\gamma}K_h)v_a&=&\hat {r}_a,\\
   (M_h+\sqrt{\gamma}K_h)v_b&=&\hat{r}_b.
  \end{eqnarray*}

\end{itemize}
In our numerical experiments, the approximation $\widehat{G}$ corresponding to the matrix $G:=M_h+\sqrt{\gamma}K_h$ is implemented by 20 steps of Chebyshev semi-iteration when the parameter $\gamma$ is small, since in this case the coefficient matrix $G$ is dominated by the mass matrix and 20 steps of Chebyshev semi-iteration is an appropriate approximation for the action of $G$'s inverse. For more details on the Chebyshev semi-iteration method we refer to {\rm\cite{ReDoWa,chebysevsemiiteration}}. Meanwhile, for the large values of $\gamma$, the stiffness matrix $K_h$ makes a significant contribution. Hence, a fixed number of Chebyshev semi-iteration is no longer sufficient to approximate the action of $G^{-1}$. In this case, the way to avoid this difficulty is to approximate the action of $G^{-1}$ with two AMG V-cycles, which obtained by the \textbf{amg} operator in the iFEM software package\footnote{\noindent \textrm{For more details about the iFEM software package, we refer to the website \url{http://www.math.uci.edu/~chenlong/programming.html} }}.

It is obvious that the {\rm PMHSS} preconditioner requires the (approximate) solution of two linear systems with coefficient matrix $M_h+\sqrt{\gamma}K_h$ at each iteration. However, if the classical ADMM is employed, we need to solve (\ref{eqn:saddle point1}) by {\rm MINRES} with block diagonal preconditioner $\mathcal{P}_{BD}$:
\begin{equation}\label{precondition matrix2}
 \mathcal{P}_{BD}= \left[
     \begin{array}{ccc}
       M_h & 0  & 0 \\
        0  & \alpha M_h+\sigma I & 0  \\
        0  & 0 & K_hM_h^{-1}K_h  \\
     \end{array}
   \right]
\end{equation}
which requires at each iteration the (approximate) solution of two linear systems with coefficient matrix $K_h$ in addition to two linear systems involving the mass matrix $M_h$ and the matrix $\alpha M_h+\sigma I$.
\subsubsection{Terminal condition}\label{terminal condition}
Let $\epsilon$ be a given accuracy tolerance. Thus we terminate our ihADMM method when $\eta<\epsilon$,
where
  $\eta=\max{\{\eta_1,\eta_2,\eta_3,\eta_4,\eta_5\}}$,
in which
\begin{equation*}
  \begin{aligned}
     & \eta_1=\frac{\|K_hy- M_hu-M_h y_c\|}{1+\|M_hy_c\|},\quad \eta_2=\frac{\|M_h(u-z)\|}{1+\|u\|},\quad
     \eta_3=\frac{\|M_h(y- y_d)+ K_hp\|}{1+\|M_hy_d\|},\\
 &\eta_4=\frac{\|\alpha M_hu -M_hp + M_h\lambda\|}{1+\|u\|},\quad
     \eta_5=\frac{\|z-{\rm\Pi}_{[a,b]}\left(z+M_h\lambda\right)\|}{1+\|z\|}.
  \end{aligned}
\end{equation*}

\subsection{A two-phase strategy for discrete problems}\label{terminal condition for PDAS}
In this section, we introduce the primal-dual active set (PDAS) method as a Phase-II algorithm to solve problem (\ref{matrix vector form}). Let $(\bar{y}^*,\bar{u}^*)$ be the optimal solution of (\ref{matrix vector form}) which can be characterized by the following optimality system that is regarded as a discretized type of (\ref{eqn:KKT2}):

\begin{equation}\label{eqn:discretizedKKT}
 G(\bar y^*,\bar u^*,\bar p^*,\bar \mu^*)=\left( \begin{aligned}
      & \qquad \qquad \qquad K_h\bar y^*-M_h\bar u^*\\
       &\qquad \qquad \qquad K_h\bar p^*+M_h(\bar y^*-y_d)\\
       &\qquad \qquad\qquad \alpha M_h\bar u^*-M_h\bar p^*+\bar \mu^*\\
       & \bar \mu^*-\max(0, \bar \mu^*+c(\bar u^*-b))-\min(0, \bar\mu^*+c(\bar u^*-a))
                          \end{aligned} \right)=0
 \end{equation}
for any $c>0$.

Thus, the full numerical scheme of the primal-dual active set method for solving the nonsmooth equation (\ref{eqn:discretizedKKT}) is shown in Algorithm \ref{algo3:Primal-Dual Active Set (PDAS) method1}.

 \begin{algorithm}[H]
  \caption{\ Primal-Dual Active Set (PDAS) method for (\ref{matrix vector form})}
  \label{algo3:Primal-Dual Active Set (PDAS) method1}
Initialization: Choose $y^0$, $u^0$, $p^0$ and $\mu^0$. Set $k=0$ and $c>0$.
\begin{description}
\item[Step 1] Determine the following subsets (Active and Inactive sets)
\begin{eqnarray*}
&&\mathcal{A}^{k+1}_{a} = \{i\in \{1,2,...,N_h\}: (\mu^k+c(u^k-a))_i<0\}, \\
&&\mathcal{A}^{k+1}_{b}= \{i\in \{1,2,...,N_h\}: (\mu^k+c(u^k-b))_i>0\}, \\
&& \mathcal{I}^{k+1}= \{1,2,...,N_h\}\backslash (\mathcal{A}^{k+1}_a \cup \mathcal{A}^{k+1}_a).
\end{eqnarray*}
\item[Step 2] solve the following system
\begin{eqnarray*}
\left\{\begin{aligned}
        &K_hy^{k+1} - M_hu^{k+1}=0, \\
  &K_hp^{k+1} + M_hy^{k+1}=M_hy_d,\\
  &\alpha M_hu^{k+1}-M_hp^{k+1}+\mu^{k+1} = 0,
 \end{aligned}\right.
\end{eqnarray*}
 where
\begin{eqnarray*}
 &u^{k+1}=\left\{\begin{aligned}
  a\quad {\rm \ on}~ \mathcal{A}^{k+1}_{a}\\
  b\quad {\rm \ on}~ \mathcal{A}^{k+1}_{b}
   \end{aligned}\right. \qquad and\quad   \mu^{k+1}=0\quad {\rm on}~ \mathcal{I}^{k+1}
 \end{eqnarray*}
  \item[Step 3] If a termination criterion is not met, set $k:=k+1$ and go to Step 1
\end{description}
\end{algorithm}

\begin{remark}
To numerically solve the linear system in step {\rm2} of Algorithm {\rm \ref{algo3:Primal-Dual Active Set (PDAS) method1}}, we partition the control $u$ according to the dimension of the sets $\mathcal{A}^{k+1}_{a}$, $\mathcal{A}^{k+1}_{b}$ and $\mathcal{I}^{k+1}$, and then attain a reduced system:
\begin{equation}\label{equ:PDASsaddle}\small
  \left[
    \begin{array}{ccc}
      M_h & 0 & K_h \\
      0 & \alpha M_h^{\mathcal{I}^{k+1},\mathcal{I}^{k+1}} & -M_h^{\mathcal{I}^{k+1},:} \\
      K_h & -M_h^{:,\mathcal{I}^{k+1}}& 0 \\
    \end{array}
  \right]\left[
           \begin{array}{c}
             y^{k+1} \\
             u^{\mathcal{I}^{k+1}} \\
             p^{k+1} \\
           \end{array}
         \right]=\left[
                   \begin{array}{c}
                     M_hy_d \\
                     -\alpha (M_h^{\mathcal{I}^{k+1},\mathcal{A}_a^{k+1}}a+ M_h^{\mathcal{I}^{k+1},\mathcal{A}_b^{k+1}}b) \\
                     M_h^{:,\mathcal{A}_a^{k+1}}a+M_h^{:,\mathcal{A}_b^{k+1}}b \\
                   \end{array}
                 \right]
\end{equation}
Once this system is solved, we could update the Lagrange multiplier $\mu$ associated with the sets $\mathcal{A}_a^{k+1}$ and $\mathcal{A}_b^{k+1}$ by
\begin{eqnarray}\label{equ:update multiplier}
  \mu^{\mathcal{A}_b^{k+1}} &=& M_h^{\mathcal{A}_b^{k+1},:}p^{k+1}-\alpha(M_h^{\mathcal{A}_b^{k+1},\mathcal{I}^{k+1}}u^{\mathcal{I}^{k+1}}
  +M_h^{\mathcal{A}_b^{k+1},\mathcal{A}_b^{k+1}}b+M_h^{\mathcal{A}_b^{k+1},\mathcal{A}_a^{k+1}}a) \\
  \mu^{\mathcal{A}_a^{k+1}} &=& M_h^{\mathcal{A}_a^{k+1},:}p^{k+1}-\alpha(M_h^{\mathcal{A}_a^{k+1},\mathcal{I}^{k+1}}u^{\mathcal{I}^{k+1}}
  +M_h^{\mathcal{A}_a^{k+1},\mathcal{A}_b^{k+1}}b+M_h^{\mathcal{A}_a^{k+1},\mathcal{A}_a^{k+1}}a)
\end{eqnarray}
Evidently, the linear system {\rm(\ref{equ:PDASsaddle})} represents a saddle point system with a $3\times3$ block structure which can be solved by employing some Krylov subspace methods with a good preconditioner. For a general survey of how to precondition saddle point problems we refer to {\rm\cite{saddlepoint}}. In particular, Rees and Stoll in {\rm\cite{ReeSto}} showed that the following block triangular preconditioners can be employed for the solution of saddle point system {\rm(\ref{equ:PDASsaddle})} with Bramble-Pasciak CG method {\rm\cite{BBcg}:}
\begin{equation*}
 \mathcal{P}_{BT}= \left[
     \begin{array}{ccc}
       A_0 & 0  & 0 \\
        0  & A_1 & 0  \\
        -K_h  &  M_h^{:,\mathcal{I}^{k+1}} & -S_0\\
     \end{array}
   \right],
\end{equation*}
In actual numerical implementations, we chose the preconditioners $A_0$ and $A_1$ to be {\rm 20} steps of Chebyshev semi-iteration to represent approximation to $M_h$ and $\alpha M_h^{\mathcal{I}^{k+1},\mathcal{I}^{k+1}}$, respectively; see {\rm\cite{ReDoWa, precondition}}. And the block $S_0:=\widehat{K_h}M_h^{-1}\widehat{K_h}$ represent a good approximation to $K_hM_h^{-1}K_h$, in which the approximation $\widehat{K_h}$ to $K_h$ is set to be two AMG V-cycles obtained by the \textbf{amg} operator in the iFEM software package.

In addition, let $\epsilon$ be a given accuracy tolerance. Thus we terminate our Phase-II algorithm (PDAS method) when $\eta<\epsilon$,
where $\eta=\max{\{\eta_1,\eta_2,\eta_3\}}$ and
\begin{equation*}
  \begin{aligned}
     & \eta_1=\frac{\|K_hy- M_hu-M_h y_c\|}{1+\|M_hy_c\|},\quad\quad\eta_2=\frac{\|M_h(y- y_d)+ K_hp\|}{1+\|M_hy_d\|}, \\
     &\eta_3=\frac{\|u-{\rm\Pi}_{[a,b]}\left((I-\alpha M_h)u+M_hp\right)\|}{1+\|u\|}.
  \end{aligned}
\end{equation*}
\end{remark}
\subsection{\textbf {Algorithms for comparison}}\label{subsection5.4}
In this section, in order to show the high efficiency of our ihADMM and the two-phase strategy, we introduce the details of a globalized version of PDAS as a comparison to solve (\ref{matrix vector form}). An important issue for the successful application of the PDAS scheme, is the use of a robust line-search method for globalization purposes. In our numerical implementation, the classical Armijo line search schemes is used. Then a globalized version of PDAS with Armijo line search is given.

In addition, as we have mentioned in Section \ref{sec:4}, instead of our ihADMM method and PDAS method, one can also apply the classical ADMM method and the linearized ADMM (LADMM) to solve problem (\ref{matrix vector form}) for the sake of numerical comparison.
Thus, in numerical implementation, we will also show the numerical results of the classical ADMM and the LADMM.

\section{Numerical Result}\label{sec:5}
In this section, we will use the following example to evaluate the numerical behaviour of our two-phase framework algorithm for the problem (\ref{matrix vector form}) and verify the theoretical error estimates given in Section \ref{sec:3}.

\subsection{Algorithmic Details}\label{subsec:5.1}
We begin by describing the algorithmic details which are common to all examples.

\textbf{Discretization.} As show in Section \ref{sec:3}, the discretization was carried out using piece-wise linear and continuous finite elements. For the case with domain $\Omega={B_{1}{(0)}}$, the unit circle in $\subseteq\mathbb{R}^2$, the polyhedral approximation is used. The assembly of mass and the stiffness matrices, as well as the lump mass matrix was left to the iFEM software package.

To present the finite element error estimates results, it is convenient to introduce the experimental order of convergence (EOC), which for some positive error functional $E(h)$ with $h> 0$ is defined as follows: let
\begin{equation}\label{EOC}
  \mathrm{EOC}:=\frac{\log {E(h_1)}-\log {E(h_2)}}{\log{h_1}-\log {h_2}}.
\end{equation}
where $h_1$ and $h_2$ denote two consecutive mesh sizes. It follows from this definition that if $E(h)=\mathcal{O}(h^{\gamma})$ then $\mathrm{EOC}\approx\gamma$.
The error functional $E(\cdot)$ investigated in the present section is given by
\begin{equation}\label{error norm}
  E_2(h):=\|u-u_h\|_{L^2{(\Omega)}}.
\end{equation}

\textbf{Initialization.} For all numerical examples, we choose $u=0$ as initialization $u^0$ for all algorithms.

\textbf{Parameter Setting.} For the classical ADMM, the LADMM and our ihADMM, the penalty parameter $\sigma$ was chosen as $\sigma=0.1 \alpha$. About the step-length $\tau$, we choose $\tau=1.618$ for the classical ADMM and LADMM, and $\tau=1$ for our ihADMM. For the PDAS method, the parameter in the active set strategy was chosen as $c=1$. For the LADMM method, we estimate an approximation for the Lipschitz constant $L$ with a backtracking method.

\textbf{Terminal Condition.} In numerical experiments, we measure the accuracy of an approximate optimal solution by using the corresponding K-K-T residual error for each algorithm. For the purpose of showing the efficiency of ihADMM, we report the numerical results obtained by running the classical ADMM method and the LADMM to compare with the results obtained by ihADMM. In this case, we terminate all the algorithms when $\eta<10^{-6}$ with the maximum number of iterations set at 500. Additionally, we also employ our two-phase strategy to obtain more accurate solution. As a comparison, a globalized version of the PDAS algorithm are also shown. In this case, we terminate the our ihADMM when $\eta<10^{-3}$ to warm-start the PDAS algorithm which is terminated when $\eta<10^{-11}$. Similarly, we terminate the PDAS algorithm with Armijo line search when $\eta<10^{-11}$.

\textbf{Computational environment.}
All our computational results are obtained by MATLAB Version 8.5(R2015a) running on a computer with
64-bit Windows 7.0 operation system, Intel(R) Core(TM) i7-5500U CPU (2.40GHz) and 8GB of memory.

\subsection{Examples}\label{subsec:5.2}
\begin{example}\label{example:1}\cite[Example 3.3]{HiPiUl}
  \begin{equation*}
     \left\{ \begin{aligned}
        &\min \limits_{(y,u)\in H^1_0(\Omega)\times L^2(\Omega)}^{}\ \ J(y,u)=\frac{1}{2}\|y-y_d\|_{L^2(\Omega)}^{2}+\frac{\alpha}{2}\|u\|_{L^2(\Omega)}^{2} \\
        &\qquad\quad{\rm s.t.}\qquad \quad-\Delta y=u,\quad \mathrm{in}\  \Omega\\
         &\qquad \qquad \quad\qquad \qquad ~y=0,\quad  \mathrm{on}\ \partial\Omega\\
         &\qquad \qquad \qquad\qquad  u\in U_{ad}=\{v(x)|a\leq v(x)\leq b, {\rm a.e }\  \mathrm{on}\ \Omega \},
                          \end{aligned} \right.
 \end{equation*}
 Here we consider the problem with $\Omega={B_{1}{(0)}}\subseteq\mathbb{R}^2$ denoting the unit circle. Furthermore, we set the desired state $y_d=(1-(x_1^2+x_2^2))x_1$, the parameters $\alpha=0.1$, $a=-0.2$ and $b=0.2$. In addition, the exact solutions of the problem is unknown in advance. Instead we use the numerical solutions computed on a grid with $h^*=2^{-10}$ as reference solutions.

 An an example, the discretized optimal control $u_h$ with $h=2^{-6}$ is displayed in Figure \ref{fig:control on h=$2^{-6}$}. In Table \ref{tab:1}, we present the error of the control $u$ w.r.t the $L^2$ norm with respect to the solution on the finest grid ($h^*=2^{-10}$) and the experimental order of convergence (EOC) for control, which both confirm the error estimate result as shown in Theorem \ref{theorem:error1}.

Numerical results for the accuracy of solution, number of iterations and cpu time obtained by our ihADMM, classical ADMM and LADMM methods are shown in Table \ref{tab:1}. As a result from Table \ref{tab:1}, we can see that our proposed ihADMM is a highly efficient for problem (\ref{matrix vector form}) in obtaining an approximate solution with medium accuracy, which compared to the classical ADMM and the LADMM in terms of in CPU time, especially when the discretization is in a fine level.
Furthermore, it should be specially mentioned that the numerical results in terms of iterations illustrate the mesh-independent performance of the ihADMM and the LADMM. However, iterations of the classical ADMM will increase with the refinement of the discretization.

In addition, in order to obtain more accurate solution, we employ our two-phase strategy. The numerical results are shown in Table \ref{tab:2}. As a comparison, numerical results obtained by the PDAS with line search are also shown in Table \ref{tab:2} to show the power and the importance of our two-phase framework. It can be observed that our two-phase strategy is faster and more efficient than the PDAS with line search in terms of the iteration numbers and CPU time.


\begin{figure}[ht]
\centering
\includegraphics[width=0.66\textwidth]{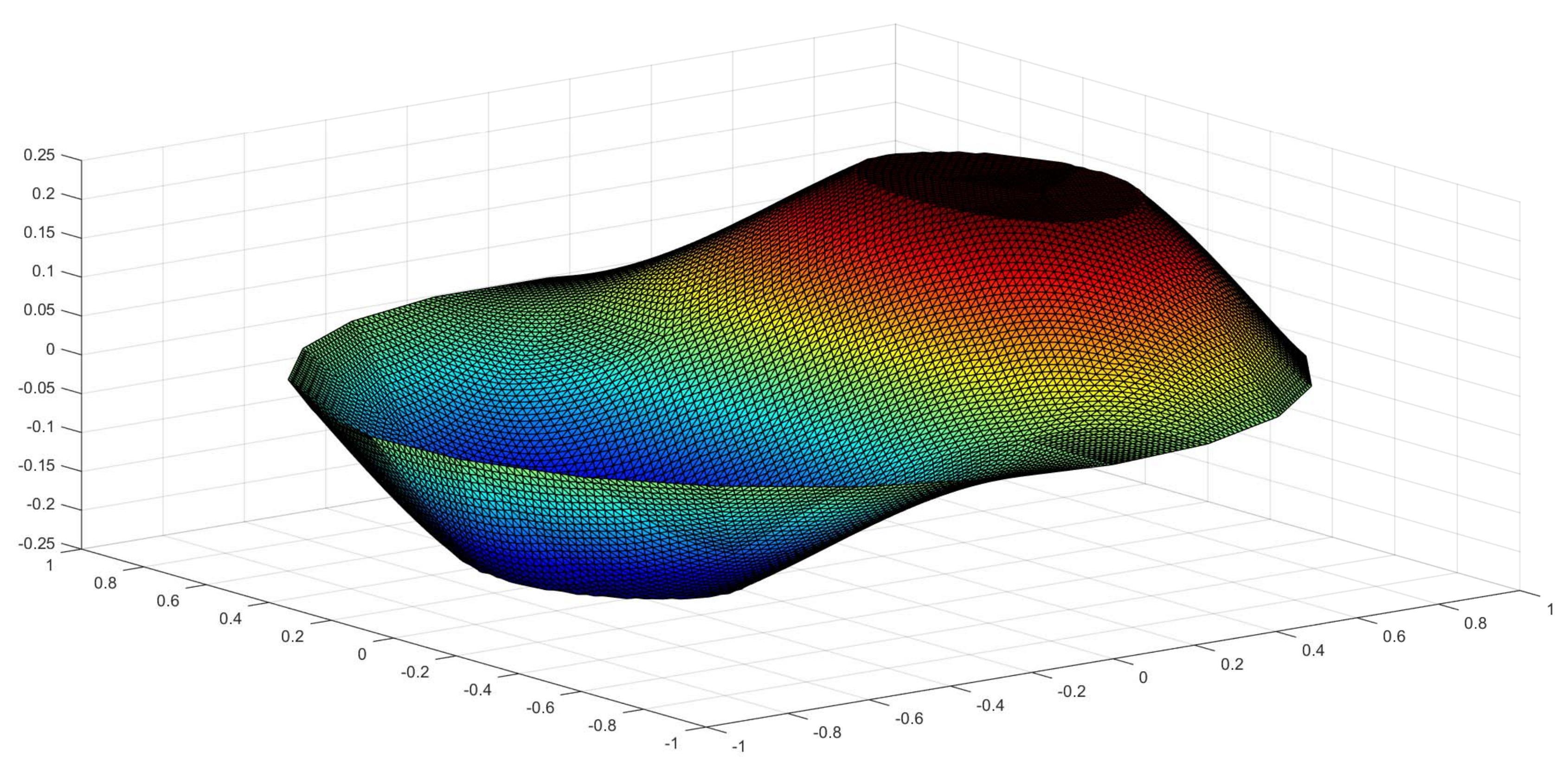}
\caption{optimal control $u_h$ on the grid with $h=2^{-6}$.}\label{fig:control on h=$2^{-6}$}
\end{figure}

\begin{table}[H]\footnotesize
\caption{Example \ref{example:1}: The convergence behavior of our ihADMM, classical ADMM and LADMM for (\ref{matrix vector form}). In the table, $\#$dofs stands for the number of degrees of freedom for the control variable on each grid level.}\label{tab:1}
\begin{center}
\begin{tabular}{@{\extracolsep{\fill}}|c|c|c|c|c|c|c|c|}
\hline
\hline
\multirow{2}{*}{$h$}&\multirow{2}{*}{$\#$dofs}&\multirow{2}{*}{$E_2$}&\multirow{2}{*}{EOC}& \multirow{2}{*}{Index} &\multirow{2}{*}{ihADMM} & \multirow{2}{*}{classical ADMM}& \multirow{2}{*}{LADMM} \\
&&&&&&&\\

\hline
                            &&&&\multirow{2}{*}{iter}            &\multirow{2}{*}{26} &\multirow{2}{*}{44} &\multirow{2}{*}{29}             \\
                            &&&&&&&\\
\multirow{2}{*}{$2^{-4}$}   &\multirow{2}{*}{148}&\multirow{2}{*}{4.09e-3}&\multirow{2}{*}{--}&\multirow{2}{*}{residual $\eta$}
                            &\multirow{2}{*}{9.13e-07}            &\multirow{2}{*}{9.91e-07}
                            &\multirow{2}{*}{9.39e-07}             \\
                            &&&&&&&\\
                            &&&&\multirow{2}{*}{CPU times/s}             &\multirow{2}{*}{0.23} &\multirow{2}{*}{0.66} &\multirow{2}{*}{0.28}  \\
                            &&&&&&&\\
\hline
                            &&&&\multirow{2}{*}{iter}            &\multirow{2}{*}{27} &\multirow{2}{*}{58} &\multirow{2}{*}{34}          \\
                            &&&&&&&\\
\multirow{2}{*}{$2^{-5}$}   &\multirow{2}{*}{635}&\multirow{2}{*}{1.46e-3}&\multirow{2}{*}{1.4858}
                            &\multirow{2}{*}{residual $\eta$}
                            &\multirow{2}{*}{8.21e-07}&\multirow{2}{*}{8.11e-07}
                            &\multirow{2}{*}{8.56e-07}            \\
                            &&&&&&&\\
                            &&&&\multirow{2}{*}{CPU time/s}          &\multirow{2}{*}{0.57} &\multirow{2}{*}{2.32} &\multirow{2}{*}{0.67}             \\
                            &&&&&&&\\
\hline
                            &&&&\multirow{2}{*}{iter}            &\multirow{2}{*}{26} &\multirow{2}{*}{76} &\multirow{2}{*}{32}             \\
                            &&&&&&&\\
\multirow{2}{*}{$2^{-6}$}   &\multirow{2}{*}{2629}&\multirow{2}{*}{4.82e-4}&\multirow{2}{*}{1.5990}&\multirow{2}{*}{residual $\eta$}
                            &\multirow{2}{*}{7.16e-07}&\multirow{2}{*}{8.10e-07}
                            &\multirow{2}{*}{8.43e-07}             \\
                            &&&&&&&\\
                            &&&&\multirow{2}{*}{CPU time/s}          &\multirow{2}{*}{1.97} &\multirow{2}{*}{9.12} &\multirow{2}{*}{2.79}            \\
                            &&&&&&&\\
\hline
                            &&&&\multirow{2}{*}{iter}            &\multirow{2}{*}{25} &\multirow{2}{*}{52} &\multirow{2}{*}{34}               \\
                            &&&&&&&\\
\multirow{2}{*}{$2^{-7}$}   &\multirow{2}{*}{10697}&\multirow{2}{*}{1.66e-4}&\multirow{2}{*}{1.5359}&\multirow{2}{*}{residual $\eta$}
                            &\multirow{2}{*}{ 5.98e-07}&\multirow{2}{*}{9.79e-07}
                            &\multirow{2}{*}{9.85e-07}           \\
                            &&&&&&&\\
                            &&&&\multirow{2}{*}{CPU time/s}          &\multirow{2}{*}{8.57} &\multirow{2}{*}{36.72} &\multirow{2}{*}{14.32}            \\
                            &&&&&&&\\
\hline
                            &&&&\multirow{2}{*}{iter}            &\multirow{2}{*}{26} &\multirow{2}{*}{130} &\multirow{2}{*}{34}              \\
                            &&&&&&&\\
\multirow{2}{*}{$2^{-8}$}   &\multirow{2}{*}{43153}&\multirow{2}{*}{7.07e-05}&\multirow{2}{*}{1.4638}
                            &\multirow{2}{*}{residual $\eta$}
                            &\multirow{2}{*}{5.97e-07}&\multirow{2}{*}{3.51e-07}
                            &\multirow{2}{*}{6.96e-07}            \\
                            &&&&&&&\\
                            &&&&\multirow{2}{*}{CPU time/s}          &\multirow{2}{*}{55.92} &\multirow{2}{*}{1303.63} &\multirow{2}{*}{82.88}            \\
                            &&&&&&&\\
\hline
                            &&&&\multirow{2}{*}{iter}            &\multirow{2}{*}{26} &\multirow{2}{*}{417} &\multirow{2}{*}{35}              \\
                            &&&&&&&\\
\multirow{2}{*}{$2^{-9}$}   &\multirow{2}{*}{173345}&\multirow{2}{*}{2.41e-05}&\multirow{2}{*}{1.4810}
                            &\multirow{2}{*}{residual $\eta$}
                            &\multirow{2}{*}{8.57e-07}&\multirow{2}{*}{9.66e-07}
                            &\multirow{2}{*}{9.48e-07}             \\
                            &&&&&&&\\
                            &&&&\multirow{2}{*}{CPU time/s}          &\multirow{2}{*}{588.76} &\multirow{2}{*}{68550.14} &\multirow{2}{*}{1080.24}           \\
                            &&&&&&&\\
\hline
&&&&\multirow{2}{*}{iter}            &\multirow{2}{*}{26} &\multirow{2}{*}{\color{red}500} &\multirow{2}{*}{35} \\
                            &&&&&&&\\
\multirow{2}{*}{$2^{-10}$}   &\multirow{2}{*}{694849}&\multirow{2}{*}{--}&\multirow{2}{*}{--}
                            &\multirow{2}{*}{residual $\eta$}
                            &\multirow{2}{*}{7.49e-07}&\multirow{2}{*}{\color{red}2.20e-05}
                            &\multirow{2}{*}{9.71e-07}             \\
                            &&&&&&&\\
                            &&&&\multirow{2}{*}{CPU time/s}          &\multirow{2}{*}{10335.43} &\multirow{2}{*}{\color{red}469845.38} &\multirow{2}{*}{15290.49} \\
                            &&&&&&&\\
\hline
\end{tabular}
\end{center}
\end{table}

\begin{table}[H]\footnotesize
\caption{Example \ref{example:1}: The convergence behavior of our two-phase strategy, PDAS with line search.}\label{tab:2}
\begin{center}
\begin{tabular}{@{\extracolsep{\fill}}|c|c|c|cc|c|c|c|}
\hline
\hline
\multirow{2}{*}{$h$}&\multirow{2}{*}{$\#$dofs}& \multirow{2}{*}{Index of performance} &\multicolumn{2}{c|}{Two-Phase strategy}& \multirow{2}{*}{PDAS with line search} \\
\cline{4-5}
& & &ihADMM\ $+$\ PDAS&& \\
\hline
                            &&\multirow{2}{*}{iter}            &\multirow{2}{*}{8\quad $+$\quad 11} &&\multirow{2}{*}{35} \\
                            &&&&&\\
\multirow{2}{*}{$2^{-4}$}   &\multirow{2}{*}{148}&\multirow{2}{*}{residual $\eta$}
                            &\multirow{2}{*}{8.84e-4\quad $/$\quad8.15e-12}           & &\multirow{2}{*}{8.17e-12}
            \\
                            &&&&&\\
                            &&\multirow{2}{*}{CPU times/s}             &\multirow{2}{*}{0.07\quad $+$\quad0.25} &&\multirow{2}{*}{0.79}  \\
                            &&&&&\\
\hline
                            &&\multirow{2}{*}{iter}            &\multirow{2}{*}{9\quad $+$\quad 12} &&\multirow{2}{*}{34} \\
                            &&&&&\\
\multirow{2}{*}{$2^{-5}$}   &\multirow{2}{*}{635}
                            &\multirow{2}{*}{residual $\eta$}
                            &\multirow{2}{*}{6.45e-04\quad $/$\quad7.49e-12}          & &\multirow{2}{*}{7.47e-12}            \\
                            &&&&&\\
                            &&\multirow{2}{*}{CPU time/s}          &\multirow{2}{*}{0.13\quad $+$\quad0.70} &&\multirow{2}{*}{1.98} \\
                            &&&&&\\
\hline
                            &&\multirow{2}{*}{iter}            &\multirow{2}{*}{8\quad $+$\quad 11} &&\multirow{2}{*}{35} \\
                            &&&&&\\
\multirow{2}{*}{$2^{-6}$}   &\multirow{2}{*}{2629}&\multirow{2}{*}{residual $\eta$}
                            &\multirow{2}{*}{7.99e-04\quad $/$\quad1.51e-12}          &  &\multirow{2}{*}{1.48e-12}             \\
                            &&&&&\\
                            &&\multirow{2}{*}{CPU time/s}          &\multirow{2}{*}{0.61\quad $+$\quad4.18} &&\multirow{2}{*}{13.31}\\
                            &&&&&\\
\hline
                            &&\multirow{2}{*}{iter}            &\multirow{2}{*}{8\quad $+$\quad 12} &&\multirow{2}{*}{34} \\
                            &&&&&\\
\multirow{2}{*}{$2^{-7}$}   &\multirow{2}{*}{10697}&\multirow{2}{*}{residual $\eta$}
                            &\multirow{2}{*}{6.61e-04\quad $/$\quad4.95e-12}        &   &\multirow{2}{*}{1.52e-12}          \\
                            &&&&&\\
                            &&\multirow{2}{*}{CPU time/s}          &\multirow{2}{*}{3.14\quad $+$\quad16.52} &&\multirow{2}{*}{52.50} \\
                            &&&&&\\
\hline
                            &&\multirow{2}{*}{iter}            &\multirow{2}{*}{8\quad $+$\quad 12} &&\multirow{2}{*}{36} \\
                            &&&&&\\
\multirow{2}{*}{$2^{-8}$}   &\multirow{2}{*}{43153}
                            &\multirow{2}{*}{residual $\eta$}
                            &\multirow{2}{*}{5.72e-04\quad $/$\quad1.51e-12}        &     &\multirow{2}{*}{1.50e-12}             \\
                            &&&&&\\
                            &&\multirow{2}{*}{CPU time/s}          &\multirow{2}{*}{23.22\quad $+$\quad77.34} &&\multirow{2}{*}{253.18}\\
                            &&&&&\\
\hline
                            &&\multirow{2}{*}{iter}            &\multirow{2}{*}{8\quad $+$\quad 11} &&\multirow{2}{*}{34} \\
                            &&&&&\\
\multirow{2}{*}{$2^{-9}$}   &\multirow{2}{*}{173345}
                            &\multirow{2}{*}{residual $\eta$}
                            &\multirow{2}{*}{5.62e-04\quad $+$\quad1.29e-12}       &     &\multirow{2}{*}{1.29e-12}            \\
                            &&&&&\\
                            &&\multirow{2}{*}{CPU time/s}&\multirow{2}{*}{181.57\quad $+$\quad473.53} &&\multirow{2}{*}{1463.63}\\
                            &&&&&\\
\hline
                            &&\multirow{2}{*}{iter}            &\multirow{2}{*}{8\quad $+$\quad 11} &&\multirow{2}{*}{34} \\
                            &&&&&\\
\multirow{2}{*}{$2^{-10}$}   &\multirow{2}{*}{694849}
                            &\multirow{2}{*}{residual $\eta$}
                            &\multirow{2}{*}{6.00e-04\quad $/$\quad1.59e-12}       &     &\multirow{2}{*}{1.60e-12}            \\
                            &&&&&\\
                            &&\multirow{2}{*}{CPU time/s}&\multirow{2}{*}{3180.13\quad $+$\quad7983.61} &&\multirow{2}{*}{22855.63}\\
                            &&&&&\\
\hline
\end{tabular}
\end{center}
\end{table}
\end{example}

\begin{example}\label{example:2}\cite [Example 4.1]{VariationPDAS}
  \begin{equation*}
     \left\{ \begin{aligned}
        &\min \limits_{(y,u)\in H^1_0(\Omega)\times L^2(\Omega)}^{}\ \ J(y,u)=\frac{1}{2}\|y-y_d\|_{L^2(\Omega)}^{2}+\frac{\alpha}{2}\|u\|_{L^2(\Omega)}^{2} \\
        &\qquad\quad{\rm s.t.}\qquad \quad-\Delta y=u,\quad \mathrm{in}\  \Omega\\
         &\qquad \qquad \quad\qquad \qquad ~y=0,\quad  \mathrm{on}\ \partial\Omega\\
         &\qquad \qquad \qquad\qquad  u\in U_{ad}=\{v(x)|a\leq v(x)\leq b, {\rm a.e }\  \mathrm{on}\ \Omega \},
                          \end{aligned} \right.
 \end{equation*}
Here, we consider the problem with control $u\in L^2(\Omega)$ on the unit square $\Omega= (0, 1)^2$ with $a=0.3$ and $b=1$. Furthermore, we set the parameters $\alpha=0.001$ and the desired state $y_d=-4\pi^2\alpha\sin{(\pi x)}\sin{(\pi y)}+\mathcal{S}r$, and $r=\min{(1,\max{(0.3,2\sin{(\pi x)}\sin{(\pi y)})})}$, where $\mathcal{S}$ denotes the solution operator associated with $-\Delta $. In addition, from the choice of parameters, it implies that $u\equiv r$ is the unique control solution to the continuous problem.

The exact control and the discretized optimal control on the grid with $h=2^{-7}$ are presented in Figure \ref{fig:subfig}. The error of the control $u$ w.r.t the $L^2$-norm and the EOC for control are presented in Table \ref{tab:3}. They also confirm that indeed the convergence rate is of order $o(h)$.

Numerical results for the accuracy of solution, number of iterations and cpu time obtained by our ihADMM, classical ADMM and LADMM methods are also shown in Table \ref{tab:3}. Experiment results show that the ADMM has evident advantage over the classical ADMM and the APG method in computing time. Furthermore, the numerical results in terms of iteration numbers also illustrate the mesh-independent performance of our ihADMM. In addition, in Table \ref{tab:4}, we give the numerical results obtained by our two-phase strategy and the PDAS method with line search. As a result from Table \ref{tab:4}, it can be observed that our two-phase strategy outperform the PDAS with line search in terms of the CPU time. These results demonstrate that our ihADMM is highly efficient in obtaining an approximate solution with moderate accuracy. Clearly, numerical results in terms of the accuracy of solution and CPU time demonstrate the efficiency and robustness of our proposed two-phase frame in obtaining accurate solutions. 

\begin{table}[ht]\footnotesize
\caption{Example \ref{example:2}: The convergence behavior of ihADMM, classical ADMM and LADMM for (\ref{matrix vector form}). In the table, $\#$dofs stands for the number of degrees of freedom for the control variable on each grid level.
}\label{tab:3}
\begin{center}
\begin{tabular}{@{\extracolsep{\fill}}|c|c|c|c|c|c|c|c|}
\hline
\hline
\multirow{2}{*}{$h$}&\multirow{2}{*}{$\#$dofs}&\multirow{2}{*}{$E_2$}&\multirow{2}{*}{EOC}& \multirow{2}{*}{Index} &\multirow{2}{*}{ihADMM} & \multirow{2}{*}{classical ADMM}& \multirow{2}{*}{LADMM} \\
&&&&&&&\\
\hline

                            &&&&\multirow{2}{*}{iter}            &\multirow{2}{*}{24} &\multirow{2}{*}{120} &\multirow{2}{*}{29}             \\
                            &&&&&&&\\
\multirow{2}{*}{$\sqrt{2}/2^{4}$}   &\multirow{2}{*}{225}&\multirow{2}{*}{0.0157}&\multirow{2}{*}{--}&\multirow{2}{*}{residual $\eta$}
                            &\multirow{2}{*}{ 8.46e-07}            &\multirow{2}{*}{9.98e-07}
                            &\multirow{2}{*}{9.81e-07}             \\
                            &&&&&&&\\
                            &&&&\multirow{2}{*}{CPU times/s}             &\multirow{2}{*}{0.28} &\multirow{2}{*}{4.72} &\multirow{2}{*}{0.42}  \\
                            &&&&&&&\\
\hline
                            &&&&\multirow{2}{*}{iter}            &\multirow{2}{*}{23} &\multirow{2}{*}{32} &\multirow{2}{*}{29}            \\
                            &&&&&&&\\
\multirow{2}{*}{$\sqrt{2}/2^{5}$}   &\multirow{2}{*}{961}&\multirow{2}{*}{5.95e-3}&\multirow{2}{*}{1.3992}
                            &\multirow{2}{*}{residual $\eta$}
                            &\multirow{2}{*}{4.88e-07}&\multirow{2}{*}{7.25e-07}
                            &\multirow{2}{*}{8.37e-07}             \\
                            &&&&&&&\\
                            &&&&\multirow{2}{*}{CPU time/s}          &\multirow{2}{*}{0.85} &\multirow{2}{*}{4.00} &\multirow{2}{*}{1.06}            \\
                            &&&&&&&\\
\hline
                            &&&&\multirow{2}{*}{iter}            &\multirow{2}{*}{23} &\multirow{2}{*}{45} &\multirow{2}{*}{31}              \\
                            &&&&&&&\\
\multirow{2}{*}{$\sqrt{2}/2^{6}$}   &\multirow{2}{*}{3969}&\multirow{2}{*}{1.89e-3}&\multirow{2}{*}{1.6558}&\multirow{2}{*}{residual $\eta$}
                            &\multirow{2}{*}{8.70e-07}&\multirow{2}{*}{3.34e-07}
                            &\multirow{2}{*}{4.53e-07}             \\
                            &&&&&&&\\
                            &&&&\multirow{2}{*}{CPU time/s}          &\multirow{2}{*}{4.81} &\multirow{2}{*}{22.35} &\multirow{2}{*}{6.94}            \\
                            &&&&&&&\\
\hline
                            &&&&\multirow{2}{*}{iter}            &\multirow{2}{*}{25} &\multirow{2}{*}{78} &\multirow{2}{*}{32}              \\
                            &&&&&&&\\
\multirow{2}{*}{$\sqrt{2}/2^{7}$}   &\multirow{2}{*}{16129}&\multirow{2}{*}{7.21e-4}&\multirow{2}{*}{1.3898}&\multirow{2}{*}{residual $\eta$}
                            &\multirow{2}{*}{4.85e-08}&\multirow{2}{*}{8.86e-08}
                            &\multirow{2}{*}{1.87e-08}             \\
                            &&&&&&&\\
                            &&&&\multirow{2}{*}{CPU time/s}          &\multirow{2}{*}{21.07} &\multirow{2}{*}{373.25} &\multirow{2}{*}{29.61}            \\
                            &&&&&&&\\
\hline
                            &&&&\multirow{2}{*}{iter}            &\multirow{2}{*}{24} &\multirow{2}{*}{183} &\multirow{2}{*}{30}              \\
                            &&&&&&&\\
\multirow{2}{*}{$\sqrt{2}/2^{8}$}   &\multirow{2}{*}{65025}&\multirow{2}{*}{2.48e-4}&\multirow{2}{*}{1.5383}
                            &\multirow{2}{*}{residual $\eta$}
                            &\multirow{2}{*}{3.66e-07}&\multirow{2}{*}{7.05e-07}
                            &\multirow{2}{*}{3.15e-07}             \\
                            &&&&&&&\\
                            &&&&\multirow{2}{*}{CPU time/s}          &\multirow{2}{*}{142.54} &\multirow{2}{*}{2669.46} &\multirow{2}{*}{191.29}            \\
                            &&&&&&&\\
\hline
                            &&&&\multirow{2}{*}{iter}            &\multirow{2}{*}{22} &\multirow{2}{*}{283} &\multirow{2}{*}{31}              \\
                            &&&&&&&\\
\multirow{2}{*}{$\sqrt{2}/2^{9}$}   &\multirow{2}{*}{261121}&\multirow{2}{*}{8.87e-05}&\multirow{2}{*}{1.4841}
                            &\multirow{2}{*}{residual $\eta$}
                            &\multirow{2}{*}{7.57e-07}&\multirow{2}{*}{5.56e-07}
                            &\multirow{2}{*}{4.81e-07}             \\
                            &&&&&&&\\
                            &&&&\multirow{2}{*}{CPU time/s}          &\multirow{2}{*}{1514.26} &\multirow{2}{*}{42758.33} &\multirow{2}{*}{2063.58}            \\
                            &&&&&&&\\
\hline
                            &&&&\multirow{2}{*}{iter}            &\multirow{2}{*}{24} &\multirow{2}{*}{\color{red}500} &\multirow{2}{*}{29}              \\
                            &&&&&&&\\
\multirow{2}{*}{$\sqrt{2}/2^{10}$}   &\multirow{2}{*}{1046529}&\multirow{2}{*}{3.15e-05}&\multirow{2}{*}{1.4936}
                            &\multirow{2}{*}{residual $\eta$}
                            &\multirow{2}{*}{5.12e-07}&\multirow{2}{*}{\color{red}4.58e-06}
                            &\multirow{2}{*}{2.87e-07}             \\
                            &&&&&&&\\
                            &&&&\multirow{2}{*}{CPU time/s}          &\multirow{2}{*}{22267.64} &\multirow{2}{*}{\color{red}545843.68} &\multirow{2}{*}{28304.33}            \\
                            &&&&&&&\\
\hline
\end{tabular}
\end{center}
\end{table}

\begin{figure}[ht]
\centering
\subfigure[exact control $u$]{
\label{fig:subfig:a}
\includegraphics[width=0.46\textwidth]{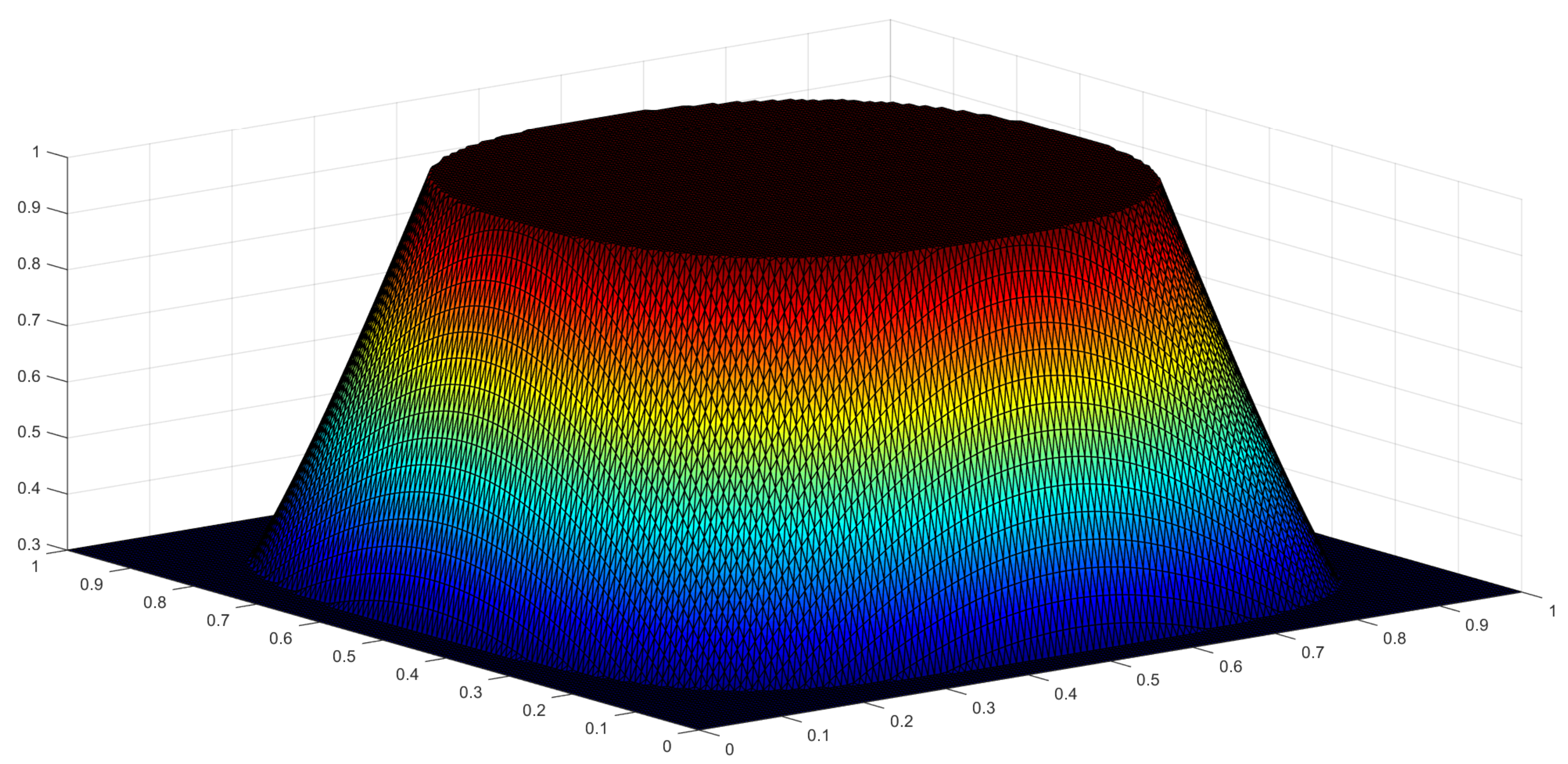}}
\subfigure[optimal control $u_h$]{
\label{fig:subfig:b}
\includegraphics[width=0.46\textwidth]{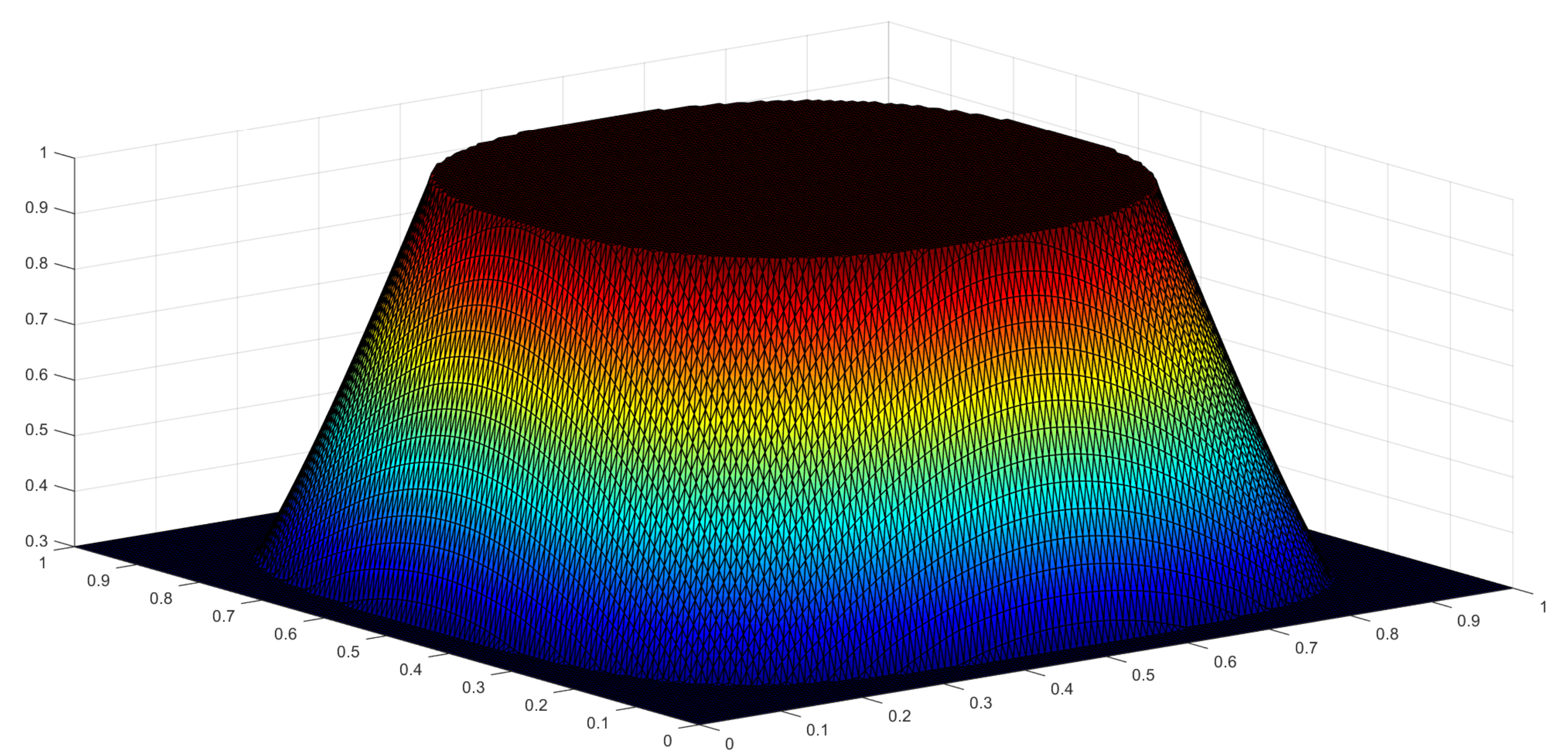}}
\caption{control solution on the grid of size $h=2^{-7}$}
\label{fig:subfig} 
\end{figure}

\begin{table}[H]\footnotesize
\caption{Example \ref{example:2}: The behavior of two-phase strategy and the PDAS method.}\label{tab:4}
\begin{center}
\begin{tabular}{@{\extracolsep{\fill}}|c|c|c|cc|c|c|c|}
\hline
\hline
\multirow{2}{*}{$h$}&\multirow{2}{*}{$\#$dofs}& \multirow{2}{*}{Index of performance} &\multicolumn{2}{c|}{Two-Phase strategy}& \multirow{2}{*}{PDAS with line search} \\
\cline{4-5}
& & &ihADMM\ $+$\ PDAS&& \\
\hline
                            &&\multirow{2}{*}{iter}            &\multirow{2}{*}{10\quad $+$\quad 10} &&\multirow{2}{*}{29} \\
                            &&&&&\\
\multirow{2}{*}{$\sqrt{2}/2^{4}$}   &\multirow{2}{*}{225}&\multirow{2}{*}{residual $\eta$}
                            &\multirow{2}{*}{8.68e-04\quad $/$\quad1.67e-12}           & &\multirow{2}{*}{1.67e-12}
            \\
                            &&&&&\\
                            &&\multirow{2}{*}{CPU times/s}             &\multirow{2}{*}{0.15\quad $+$\quad0.80} &&\multirow{2}{*}{2.44}  \\
                            &&&&&\\
\hline
                            &&\multirow{2}{*}{iter}            &\multirow{2}{*}{12\quad $+$\quad 9} &&\multirow{2}{*}{27} \\
                            &&&&&\\
\multirow{2}{*}{$\sqrt{2}/2^{5}$}   &\multirow{2}{*}{961}
                            &\multirow{2}{*}{residual $\eta$}
                            &\multirow{2}{*}{6.58e-04\quad $/$\quad3.95e-12}          & &\multirow{2}{*}{3.95e-12}            \\
                            &&&&&\\
                            &&\multirow{2}{*}{CPU time/s}          &\multirow{2}{*}{0.56\quad $+$\quad2.36} &&\multirow{2}{*}{6.52} \\
                            &&&&&\\
\hline
                            &&\multirow{2}{*}{iter}            &\multirow{2}{*}{11\quad $+$\quad 8} &&\multirow{2}{*}{30} \\
                            &&&&&\\
\multirow{2}{*}{$\sqrt{2}/2^{6}$}   &\multirow{2}{*}{3969}&\multirow{2}{*}{residual $\eta$}
                            &\multirow{2}{*}{8.24e-04\quad $/$\quad5.73e-12}          &  &\multirow{2}{*}{ 5.73e-12}             \\
                            &&&&&\\
                            &&\multirow{2}{*}{CPU time/s}          &\multirow{2}{*}{2.84\quad $+$\quad8.14} &&\multirow{2}{*}{26.99}\\
                            &&&&&\\
\hline
                            &&\multirow{2}{*}{iter}            &\multirow{2}{*}{11\quad $+$\quad 10} &&\multirow{2}{*}{29} \\
                            &&&&&\\
\multirow{2}{*}{$\sqrt{2}/2^{7}$}   &\multirow{2}{*}{16129}&\multirow{2}{*}{residual $\eta$}
                            &\multirow{2}{*}{8.91e-04\quad $/$\quad1.42e-12}        &   &\multirow{2}{*}{1.41e-12}          \\
                            &&&&&\\
                            &&\multirow{2}{*}{CPU time/s}          &\multirow{2}{*}{12.70\quad $+$\quad48.06} &&\multirow{2}{*}{143.79} \\
                            &&&&&\\
\hline
                            &&\multirow{2}{*}{iter}            &\multirow{2}{*}{10\quad $+$\quad 12} &&\multirow{2}{*}{31} \\
                            &&&&&\\
\multirow{2}{*}{$\sqrt{2}/2^{8}$}   &\multirow{2}{*}{65025}
                            &\multirow{2}{*}{residual $\eta$}
                            &\multirow{2}{*}{9.90e-04\quad $/$\quad8.26e-13}        &     &\multirow{2}{*}{8.25e-13}             \\
                            &&&&&\\
                            &&\multirow{2}{*}{CPU time/s}          &\multirow{2}{*}{83.14\quad $+$\quad225.36} &&\multirow{2}{*}{757.71}\\
                            &&&&&\\
\hline
                            &&\multirow{2}{*}{iter}            &\multirow{2}{*}{11\quad $+$\quad 10} &&\multirow{2}{*}{30} \\
                            &&&&&\\
\multirow{2}{*}{$\sqrt{2}/2^{9}$}   &\multirow{2}{*}{261121}
                            &\multirow{2}{*}{residual $\eta$}
                            &\multirow{2}{*}{6.57e-04\quad $/$\quad2.55e-12}       &     &\multirow{2}{*}{2.55e-12}            \\
                            &&&&&\\
                            &&\multirow{2}{*}{CPU time/s}&\multirow{2}{*}{887.38\quad $+$\quad1886.33} &&\multirow{2}{*}{4629.40}\\
                            &&&&&\\
\hline
                            &&\multirow{2}{*}{iter}            &\multirow{2}{*}{11\quad $+$\quad 10} &&\multirow{2}{*}{31} \\
                            &&&&&\\
\multirow{2}{*}{$\sqrt{2}/2^{10}$}   &\multirow{2}{*}{1046529}
                            &\multirow{2}{*}{residual $\eta$}
                            &\multirow{2}{*}{9.25e-04\quad $/$\quad1.04e-13}       &     &\multirow{2}{*}{1.04e-13}            \\
                            &&&&&\\
                            &&\multirow{2}{*}{CPU time/s}&\multirow{2}{*}{10206.22\quad $+$\quad12927.28} &&\multirow{2}{*}{38760.16}\\
                            &&&&&\\
\hline
\end{tabular}
\end{center}
\end{table}

\end{example}

\section{Concluding Remarks}\label{sec:6}
In this paper, we have designed a two-phase method for solving the optimal control problems with box control constraints. By taking advantage of inherent structures of the problem, in Phase-I, we proposed an inexact heterogeneous ADMM (ihADMM) to solve discretized problems. Furthermore, theoretical results on the global convergence as well as the iteration complexity results $o(1/k)$ in non-ergodic sense for ihADMM were established. Moreover, an implementable inexactness criteria was used which allow the accuracy of the generated ihADMM to be easily implementable. Moreover, in order to obtain more accurate solution, in Phase-II taking the advantage of the local superlinear convergence of the primal dual active set method, the PDAS method is used as a postprocessor of the ihADMM. Numerical results demonstrated the efficiency of our ihADMM and the two-phase strategy.
\appendix
\section{Proof of Proposition \ref{descent proposition} and Proposition \ref{descent proposition2}}
\label{sec:proof}
Before giving two qusi-descent properties: e.g., the proof of Proposition \ref{descent proposition} and Proposition \ref{descent proposition2}, we first introduce the following two basic identities:
\begin{eqnarray}
 \label{identity1} &&\langle x, y\rangle_Q =\frac{1}{2}(\|x\|^2_Q+\|y\|^2_Q-\|x-y\|^2_Q)=\frac{1}{2}(\|x+y\|^2_Q-\|x\|^2_Q-\|y\|^2_Q) \\
  \label{identity2}&&\langle x-x', y-y'\rangle_Q  = \frac{1}{2}(\|x+y\|^2_Q+\|x'+y'\|^2_Q-\|x+y'\|^2_Q-\|x'+y\|^2_Q)
\end{eqnarray}
which hold for any $x, y, x', y'$ in the same Euclidean space and a self-adjoint positive semidefinite linear operator $Q$. The two identities would be frequently used in proof of Proposition \ref{descent proposition} and Proposition \ref{descent proposition2}.

\begin{proof}
Based on the optimality condition (\ref{equ: inexact variational inequalities1}) and (\ref{equ: inexact variational inequalities2}) for $(u^{k+1}, z^{k+1})$ and the optimality condition (\ref{equ: exact variational inequalities1}) and (\ref{equ: exact variational inequalities2}) for $(u^*,z^*)$, let $u_1=u^{k+1}, u_2=u^*, z_1=z^{k+1}$ and $z_2=z^*$ in (\ref{subdifferential strongly monotone}) and (\ref{subdifferential monotone}), respectively, we are able to derive that
\begin{eqnarray}\label{variational inequlaities}
\label{variational inequlaities1}&&\langle\delta^k-(M_h\lambda^k+\sigma M_h(u^{k+1}-z^k))+M_h\lambda^*, u^{k+1}-u^*\rangle\geq\|u^{k+1}-u^*\|^2_{\Sigma_f}, \\
\label{variational inequlaities2}&&\langle M_h\lambda^k+\sigma W_h(u^{k+1}-z^k))-M_h\lambda^*, z^{k+1}-z^*\rangle\geq\frac{\sigma}{2}\|z^{k+1}-z^*\|^2_{W_h}.
\end{eqnarray}
Adding (\ref{variational inequlaities1}) and (\ref{variational inequlaities2}), we get
\begin{equation}\label{inequlaities1}
  \begin{aligned}
&\langle\delta^k,u^{k+1}-u^*\rangle-\langle\tilde\lambda^{k+1}-\lambda^*,M_hr^{k+1}\rangle-\sigma\langle M_h(z^{k+1}-z^k),u^{k+1}-u^*\rangle\\
&+\langle r^{k+1}, (W_h-M_h)(z^{k+1}-z^*)\rangle\geq\|u^{k+1}-u^*\|^2_{\Sigma_f}+\frac{\sigma}{2}\|z^{k+1}-z^*\|^2_{W_h},
  \end{aligned}
\end{equation}
where we have used the fact that $\lambda^k+\sigma(u^{k+1}-z^k)=\tilde\lambda^{k+1}+\sigma(z^{k+1}-z^k)$ and $u^*=z^*$. Next, we rewrite the last three terms on the left-hand side of (\ref{inequlaities1}). First, by (\ref{identity1}), we have that
\begin{equation}\label{descent estimates1}
  \begin{aligned}
  \langle\lambda^*-\tilde\lambda^{k+1},M_hr^{k+1}\rangle&=\langle\lambda^*-\lambda^k-\sigma r^{k+1}, M_hr^{k+1}\rangle \\
  &=\frac{1}{\tau\sigma}\langle\lambda^*-\lambda^k, M_h(\lambda^{k+1}-\lambda^k)\rangle -\sigma\|r^{k+1}\|^2_{M_h} \\
   &=\frac{1}{2\tau\sigma}(\|\lambda^k-\lambda^*\|^2_{M_h}-\|\lambda^{k+1}-\lambda^*\|^2_{M_h})+\frac{(\tau-2)\sigma}{2}\|r^{k+1}\|^2_{M_h}.
  \end{aligned}
\end{equation}
Second, by employing (\ref{identity2}) and $u^*=z^*$, we have
\begin{equation}\label{descent estimates2}
  \begin{aligned}
   \sigma\langle M_h(z^{k+1}-z^k),u^*-u^{k+1}\rangle =&\frac{\sigma}{2}\|z^k-z^*\|^2_{M_h}+\frac{\sigma}{2}\|r^{k+1}\|^2_{M_h}\\
   &-\frac{\sigma}{2}\|z^{k+1}-z^*\|^2_{M_h}-\frac{\sigma}{2}\|u^{k+1}-z^{k}\|^2_{M_h}.
  \end{aligned}
\end{equation}
Third, by Proposition \ref{eqn:martix properties}, we know $W_h-M_h$ is a symmetric positive definite matrix. Then using (\ref{identity1}) and $u^*=z^*$ , we get
\begin{equation}\label{descent estimates3}
  \begin{aligned}
&\langle r^{k+1}, (W_h-M_h)(z^{k+1}-z^*)\rangle\\
=&\langle u^{k+1}-z^{k+1}, (W_h-M_h)(z^{k+1}-z^*)\rangle\\
=&\frac{\sigma}{2}\|u^{k+1}-u^*\|^2_{W_h-M_h}-\frac{\sigma}{2}\|z^{k+1}-z^*\|^2_{W_h-M_h}-\frac{\sigma}{2}\|r^{k+1}\|^2_{W_h-M_h}.
  \end{aligned}
\end{equation}
Then, substituting (\ref{descent estimates1}), (\ref{descent estimates2}) and (\ref{descent estimates3}) into (\ref{inequlaities1}), we can get the required inequality (\ref{inequlaities property1}). This completes the proof of Proposition \ref{descent proposition}. For the proof of Proposition \ref{descent proposition2}, by substituting $\bar u^{k+1}$ and $\bar z^{k+1}$ for $u^{k+1}$ and $z^{k+1}$ in the proof of Proposition \ref{descent proposition}, we can get the assertion of this proposition.
\end{proof}

\section*{Acknowledgments}
The authors would like to thank Dr. Long Chen for the FEM package iFEM \cite{Chen} in Matlab and also would like to thank the colleagues for their valuable suggestions that led to
improvement in the paper.

\end{document}